\newcommand{\abs}[1]{\left\vert#1\right\vert}

\documentclass{amsart}

\newtheorem{theorem}{Theorem}[section]
\newtheorem{lemma}[theorem]{Lemma}
\newtheorem{case}{Case}
\theoremstyle{definition}
\newtheorem{definition}[theorem]{Definition}
\newtheorem{corollary}[theorem]{Corollary}

\usepackage[top=1.5in, left=1in, right=1in, bottom=1in]{geometry}
\usepackage{url}
\usepackage[utf8]{inputenc}
\usepackage{graphicx}
\usepackage{gensymb}
\usepackage{soul}
\usepackage{tikz}
\usepackage[mathlines]{lineno}
\usepackage{scrextend}
\usepackage{xcolor}
\usetikzlibrary[shapes.geometric]
\usepackage{caption}
\usepackage{amssymb,amsmath,amsthm}

\title{Landscapes of the Octahedron}
\author{Emiko Saso, Houston Schuerger, and Xin Shi}

\begin{document}

\maketitle
\begin{abstract}
The landscapes of a polyhedron are subsets of its nets one must consider to identify all shortest paths.  Landscapes of cubes and tetrahedra have been used to identify coordinate based formulas for the lengths of the shortest paths between points on these surfaces.  We extend these results to develop formulas for the lengths of the shortest paths between points on the surface of octahedra.
\end{abstract}
\section{Introduction}

Geodesics are a subject of much interest in geometry. This area focuses on minimizing the length of paths between two points on a surface. Alexandrov's Star Unfolding is an important discovery in this area \cite{Alexandrov}, which can be used to determine the shortest distance between any two points on the surface of a convex polyhedron.  It does so by using the star polygon, which is obtained by cutting the surface of a polyhedron along the paths between a fixed initial point $s$ and each vertex of the polyhedron, then unfolding the resulting cut surface onto a plane.  Finally, the shortest path between $s$ and any arbitrarily chosen second point $p$ is the path of shortest length between $p$ and one of the copies of $s$ created during the cutting process.

Dijkstra's Algorithm offers an alternative approach to computing the shortest paths on polyhedral surfaces, extending its applicability to a larger class of polyhedra \cite{Cook}. By applying this method, authors have provided solutions to problems involving optimizing pathways in various fields, including determination of robotic motion, route planning, telecommunications and transportation on polyhedral surfaces  \cite{Agarwal, Kanai, Chen, Magnanti}. Although the methods mentioned above can yield rapid results when executed with a computer, manual calculations are notably challenging and tend to introduce a significant margin of error. In contrast, this paper offers a solution by introducing coordinate-based formulas. This shifts the computational burden away from the end user so individuals seeking to apply our findings can easily perform numerical calculations.

In \cite{Fontenot}, Fontenot et. al. introduced the concepts of landscapes and orientations of polyhedra. A landscape is the union of a sequence of consecutive faces that the shortest path between two points can lie on, and is a subset of a net of the relevant polyhedron. An orientation is an embedding of a landscape into the Cartesian plane, thus allowing the use of formulas by transforming a synthetic object (or one without coordinates) into an analytical object (or one with a coordinate structure). Having developed a collection of landscapes and orientations, they introduced a coordinate system for the surfaces of convex polyhedra, and by using the tools above, determined formulas for the distances between points in each of these landscapes.  Finally, they identified the minimal collection of landscapes necessary to contain all of the shortest paths, which they called the valid landscapes.

We will apply the concepts introduced by Fontenot et. al. to the surface of octahedra, thus expanding this body of work. To achieve this goal, we will utilize their coordinate system and net substructures in order to find where these shortest paths can exist, thus identifying the collection of valid landscapes of the octahedron. The end result will be that for any pair of points on the surface of the octahedron, we will determine the lengths and locations of the shortest paths between them.

\section{Definitions}

We will adapt the approach used by Fontenot et. al. with the tetrahedron and cube in \cite{Fontenot}. In addition, all definitions given in this section were introduced in \cite{Fontenot}.  Let $\mathcal P_8$ be a unit octahedron, that is a regular octahedron for which all edges are of length 1.  To do any calculations concerning paths on the surface of $\mathcal P_8$, we will first fix a net for $\mathcal P_8$ and define a coordinate system for the set of points on the surface of $\mathcal P_8$.  Let the faces of $\mathcal P_8$ be labeled with elements of the set $\{F_1, F_2, F_3, F_4, F_5, F_6, F_7, F_8\}$ and the vertices be represented by elements of the set
\[\big{\{}\{1,4,6,7\}, \{1,2,5,6\}, \{1,2,3,4\}, \{2,3,5,8\}, \{3,4,7,8\}, \{5,6,7,8\}\big{\}}\] 
such that face $F_n$ is incident to vertex $\{n_1,n_2,n_3,n_4\}$ if and only if $n \in \{n_1,n_2,n_3,n_4\}$. Consider the net in Figure \ref{octnet}, and thus the relative locations of the faces of $\mathcal P_8$, to be fixed for the remainder of our discussion.

\begin{figure}[h]
  \centering
\begin{tikzpicture}[scale=1]

\filldraw[fill=black,draw=black] (0,0) circle (2pt);
\filldraw[fill=black,draw=black] (4,0) circle (2pt);
\filldraw[fill=black,draw=black] (8,0) circle (2pt);
\filldraw[fill=black,draw=black] (10,3.464) circle (2pt);
\filldraw[fill=black,draw=black] (12,0) circle (2pt);
\filldraw[fill=black,draw=black] (14,-3.464) circle (2pt);
\filldraw[fill=black,draw=black] (10,-3.464) circle (2pt);
\filldraw[fill=black,draw=black] (6,-3.464) circle (2pt);
\filldraw[fill=black,draw=black] (4,-6.928) circle (2pt);
\filldraw[fill=black,draw=black] (2,-3.464) circle (2pt);
\draw[line width = 0.05mm] (0,0) -- (12,0) -- (14,-3.464) -- (2,-3.464) -- (0,0);
\draw[line width = 0.05mm] (2,-3.464) -- (4,-6.928) -- (10,3.464) -- (12,0);
\draw[line width = 0.05mm] (2,-3.464) -- (4,0) -- (6,-3.464);
\draw[line width = 0.05mm] (8,0) -- (10,-3.464) -- (12,0);
\node[above, outer sep=55] at (2,-3.464) {$F_4$};
\node[below, outer sep=55] at (4,0) {$F_1$};
\node[above, outer sep=55] at (6,-3.464) {$F_6$};
\node[below, outer sep=55] at (8,0) {$F_5$};
\node[above, outer sep=55] at (10,-3.464) {$F_8$};
\node[below, outer sep=55] at (12,0) {$F_3$};
\node[above, outer sep=55] at (4,-6.928) {$F_2$};
\node[below, outer sep=55] at (10,3.464) {$F_7$};
\node[above][font = {\small}] at (0,0) {$\{3,4,7,8\}$};
\node[above][font = {\small}] at (4,0) {$\{1,4,6,7\}$};
\node[font = {\small}] at (7.2,.27) {$\{5,6,7,8\}$};
\node[right][font = {\small}] at (10,3.464) {$\{1,4,6,7\}$};
\node[right][font = {\small}] at (12,0) {$\{3,4,7,8\}$};
\node[below][font = {\small}] at (14,-3.464) {$\{1,2,3,4\}$};
\node[below][font = {\small}] at (10,-3.464) {$\{2,3,5,8\}$};
\node[font = {\small}] at (6.8,-3.734) {$\{1,2,5,6\}$};
\node[right][font = {\small}] at (4,-6.928) {$\{2,3,5,8\}$};
\node[left][font = {\small}] at (2,-3.464) {$\{1,2,3,4\}$};

\end{tikzpicture}
    \caption{Labeling of the Octahedron}
    \label{octnet}
\end{figure}
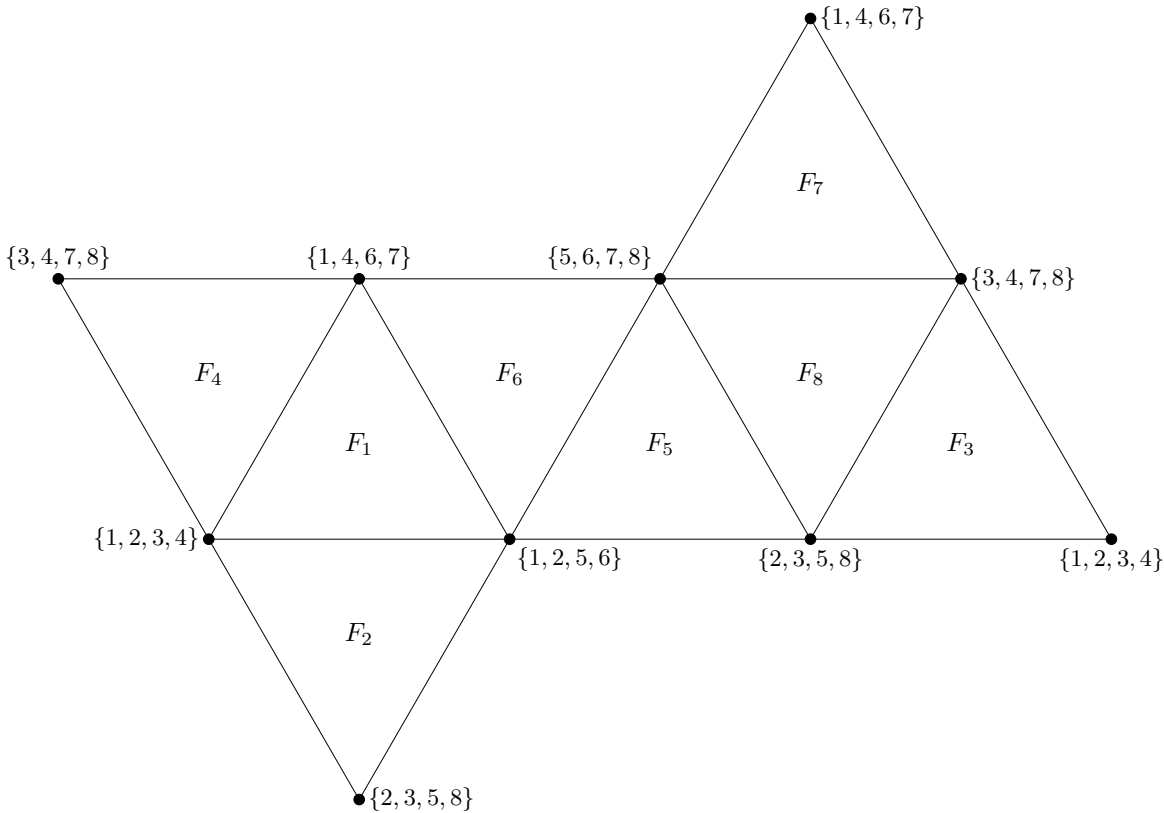
 
Let $F_n$ be a face of $\mathcal P_8$ and $g:F_n \rightarrow \mathbb R^2$ be a function mapping $F_n$ into the first quadrant, with an edge $e$ of $F_n$ mapped to the line segment $\overline{(0,0),(1,0)}$, preserving the distance between points of $F_n$ such that the exterior side of $F_n$ is facing upwards towards the viewer.  Given a point $p$ on $\mathcal P_8$, if $g(p)=(x_0,y_0)$, $p$ is an element of $F_n$, and $F_m$ is the other face of $\mathcal P_8$ incident to $e$, then $p$ will be said to have the ordered quadruple $(F_n, F_m, x_0, y_0)$ as a {\em representation}.  In addition, the point can be said to have {\em home-face} $F_n$, {\em shared-face} $F_m$, {\em $x$-coordinate} $x_0$, and {\em $y$-coordinate} $y_0$.

Let $N$ be a net of $\mathcal P_8$.  Let $G$ be the graph with vertices $\{F_i\}_{i=1}^8$ such that given $F_i,F_j$ with $1 \leq i \neq j \leq 8$ we have $F_iF_j \in E(G)$ if and only if $F_i$ and $F_j$ share an edge in $N$.  Then $G$ is the {\em dual graph} of $N$. Given a subgraph $H$ of $G$, if $H$ is a path of at least two vertices, then the union of the collection of faces of $N$ given by the vertices of $H$ form a {\em landscape} $L$, with $H$ as its dual graph.  If $F_{n_1}$ and $F_{n_2}$ are the vertices of $H$ with degree 1, then we say $L$ is a landscape from {\em origin face} $F_{n_1}$ to {\em destination face} $F_{n_2}$ and denote this by $L(F_{n_1} \rightarrow F_{n_2})$.  In addition, given a (proper) subgraph $K$ of $H$, if $K$ is a path of at least two vertices, then the union of the collection of faces of $N$ given by the vertices of $K$ form a {\em (proper) sublandscape} $L'$ of $L$.

Let $L$ be a landscape of $\mathcal P_8$, $F_n$ be contained in $L$, $F_m$ be a face of $\mathcal P_8$ adjacent to $F_n$, and $e$ be the edge shared by $F_n$ and $F_m$. 
Let $O=(L, F_n, F_m)$ reference the subset of $\mathbb R^2$ contained in the image of $L$ under the mapping given by $g:L \rightarrow \mathbb R^2$ which maps $F_n$ into the upper half-plane, $e$ to the line segment $\overline{(0,0),(1,0)}$, and preserves the distance between points in each face of $L$ such that the exterior side of each face of $L$ faces upwards towards the viewer. Refer to $O$ as an {\em orientation} of $\mathcal P_8$ and $g$ as {\em the function defining $O$}.  When we wish to view a point $p$ as an element of the orientation $(L, F_n, F_m)$ we will denote it by $p(L, F_n, F_m)$ and when we wish to denote its $x$-coordinate or $y$-coordinate we will denote them by $p(L, F_n, F_m)_x$ and $p(L, F_n, F_m)_y$ respectively.  Given a representation $(F_n,F_m,x_0,y_0)$ for a point $p$, we will refer to $(x_0,y_0)$ as the {\em standard position} of $p$ with respect to $(F_n,F_m,x_0,y_0)$.

If $p_1, p_2 \in O$ and the line segment $\overline{p_1p_2} \subset O$, then we call said line segment the {\em trail} from $p_1$ to $p_2$ in $O$ and denote it $T_O(p_1,p_2)$. If $\overline{p_1p_2} \subset O=(L,F_n,F_m)$, then we denote the length of $T_O(p_1,p_2)$ by $\abs{T_L(p_1,p_2)}$.  On the other hand, if $p_1,p_2 \in O$, but $\overline{p_1p_2} \not \subset O$ then we will set $\abs{T_L(p_1,p_2)}=\infty$.  Note that since the function $g:L \rightarrow O$ defining $O$ is distance preserving, $\abs{T_L(p_1,p_2)}$ is the length of the line segment between the points $g^{-1}(p_1)$ and $g^{-1}(p_2)$ in $L$, provided the line segment is contained in $L$.

The following theorem, introduced by Fontenot et. al., guarantees that considering the landscapes of $\mathcal P_8$ is sufficient for identifying the collection of all shortest paths on the surface of $\mathcal P_8$ .

\begin{theorem}\cite{Fontenot}
Let $\mathcal P$ be a convex unit polyhedron, $F_n$ and $F_m$ be two distinct faces of $\mathcal P$, and $p_1 \in F_n\setminus F_m$ and $p_2 \in F_m \setminus F_n$ be two distinct points on the surface of $\mathcal P$.  Then there exists an orientation $O=(L, F_n, F_m)$ of $\mathcal P$ such that the shortest path between $p_1$ and $p_2$ is $T_O(p_1, p_2)$.  
\end{theorem}

Since any fixed polyhedron $\mathcal P$ can have at most finitely many orientations, in \cite{Fontenot}, the concept of {\em surface distance} between points $p_1$ and $p_2$, denoted $d_{\mathcal P}(p_1,p_2)$, was defined to be
\[d_{\mathcal P}(p_1,p_2)=\min\left\{\abs{T_L(p_1,p_2)}: L ~\text{is a landscape of}~\mathcal P\right\}.\] 

\noindent In the next section, we will identify the minimal collection of landscapes in which all paths providing surface distances on $\mathcal P_8$ are contained. In order to do that, we need to consider an important definition.

\begin{definition}\cite{Fontenot}
Given a landscape $L_i$ of a convex unit polyhedron $\mathcal P$, $L_i$ is said to be a {\em valid} landscape of $\mathcal P$ if there exist points $p_1, p_2 \in L_i$ such that for all other landscapes $L_j$ of $\mathcal P$ with $p_1, p_2 \in L_j$, 
$\abs{T_{L_i}(p_1,p_2)} \leq \abs{T_{L_j}(p_1,p_2)}$ and if $\abs{T_{L_i}(p_1,p_2)} = \abs{T_{L_j}(p_1,p_2)}$ then $L_i$ does not contain $L_j$ as a sublandscape.
\end{definition}

\section{Landscapes of an Octahedron}
First, if a point is a vertex of the octahedron or lies on an edge of the octahedron  it may have multiple home-faces.  In particular, the vertex $\{n_1,n_2,n_3,n_4\}$ can be represented as $(F_{n_1}, F_{n_2}, 0, 0)$, $(F_{n_2}, F_{n_3}, 0, 0)$, $(F_{n_3}, F_{n_4}, 0, 0)$, and $(F_{n_4}, F_{n_1}, 0, 0)$, with the remaining vertices of the octahedron similarly having such representations. Next note, if a point lies on an edge of $\mathcal P_8$ say, $\overline {\{n_1,n_2,n_3,n_4\} \{n_1,n_2,n_5,n_6\}}$, and is represented as $(F_{n_1}, F_{n_2}, x, 0)$ then it can also be represented as $(F_{n_2}, F_{n_1}, 1-x, 0)$.  While this addresses one's ability to switch between representations using different home-faces for a point which is incident to more than one face, we still need to address how one switches between the representations of a point using different shared-faces.  Since each face of a convex unit tetrahedron and of $\mathcal P_8$ are equilateral triangles, a small modification of a lemma introduced by Fontenot et. al. provides us with the following.

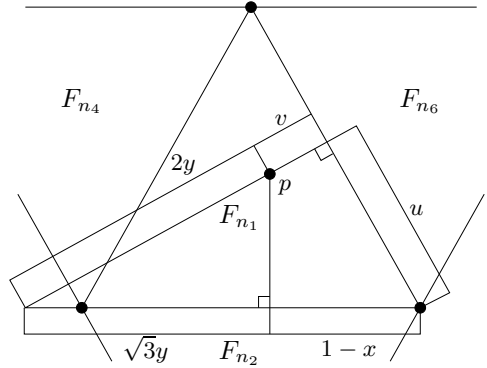
\begin{figure}[h]
  \centering
\begin{tikzpicture}[scale=1]
    \filldraw[fill=black,draw=black] (3,4) circle (2pt);
    \filldraw[fill=black,draw=black] (.75,0) circle (2pt);
    \filldraw[fill=black,draw=black] (5.25,0) circle (2pt);
    \filldraw[fill=black,draw=black] (3.25,1.78) circle (2pt);
    \draw[line width = 0.025mm] (0,4) -- (6,4);
    \draw[line width = 0.025mm] (3,4) -- (5.25,0) -- (.75,0) -- (3,4);
    \draw[line width = 0.025mm] (.75,0) -- (0,0) -- (4,2.2);
    \draw[line width = 0.025mm] (3.25,1.78) -- (3.25,-.35);
    \draw[line width = 0.025mm] (-.015,0) -- (-.015,-.35) -- (5.25,-.35) -- (5.25,0);
    \draw[line width = 0.025mm] (3.1,0) -- (3.1,.15) -- (3.25,.15);
    \node[font = {\normalsize}] at (.75,2.75) {$F_{n_4}$};
    \node[font = {\normalsize}] at (5.25,2.75) {$F_{n_6}$};
    \node[font = {\small}] at (3.45,1.6) {$p$};
    \node[font = {\small}] at (4.3,-.55) {$1-x$};
    \node[font = {\small}] at (1.6,-.55) {$\sqrt{3}y$};
    \node[font = {\normalsize}] at (2.85,-.6) {$F_{n_2}$};
    \node[font = {\normalsize}] at (2.85,1.2) {$F_{n_1}$};
\usetikzlibrary{calc}
\newcommand\rightAngle[4]{
  \pgfmathanglebetweenpoints{\pgfpointanchor{#2}{center}}{\pgfpointanchor{#1}{center}}
  \coordinate (tmpRA) at ($(#2)+(\pgfmathresult+45:#4)$);
  \draw[black,thin] ($(#2)!(tmpRA)!(#1)$) -- (tmpRA) -- ($(#2)!(tmpRA)!(#3)$);
}
  \coordinate (O) at (4,2.2);
  \coordinate (X) at (0,0);
  \coordinate (Y) at (5.25,0);
  \rightAngle{X}{O}{Y}{0.25}
  \draw[line width = 0.05mm] (4,2.2) -- (4.4,2.42) -- (5.640453489,.2) -- (5.25,0);
\node[font = {\small}] at (5.2,1.3) {$u$};
  \draw[line width = 0.05mm] (3.8,2.5777) -- (-.2068669528,.368223176) -- (0,0);
\node[font = {\small}] at (3.4,2.5) {$v$};
\node[font = {\small}] at (2.1,1.9) {$2y$};
  \draw[line width = 0.05mm] (3.25,1.78) -- (3.039914163,2.15395279);
  \draw[line width = 0.05mm] (1.15,-.7114) -- (-.1,1.511);
  \draw[line width = 0.05mm] (4.85,-.7114) -- (6.1,1.51111);
\end{tikzpicture}
\caption{Representations of a Point $p$ with Different Shared-Faces, where $u=\frac{1-x+\sqrt{3}y}{2}$ and $v = \frac{\sqrt{3}-\sqrt{3}x-y}{2}$}
\end{figure}

\begin{lemma}\label{octahedronrepresentation}\cite{Fontenot}
Suppose $\{n_1,n_2,n_3,n_4,n_5,n_6,n_7,n_8\}$ $=$ $\{1,2,3,4,5,6,7,8\}$ such that in nets of our copy of $\mathcal P_8$ the faces $F_{n_1}$ and $F_{n_8}$ are opposite and $F_{n_2}$, $F_{n_4}$, and $F_{n_6}$ occur about $F_{n_1}$ in the same counter-clockwise order as $F_2$, $F_4$, and $F_6$ occur about $F_1$. If a point $p\in\mathcal P_8$ can be represented as $(F_{n_1}, F_{n_2}, x, y)$, then it can also be represented as 
\[\left(F_{n_1}, F_{n_6}, \frac {1-x+\sqrt{3}y}{2}, \frac{\sqrt{3}-\sqrt{3}x-y}{2}\right).\] 

\end{lemma}

Having the concepts needed, we now turn our focus to determining just how many valid landscapes an octahedron has.  To construct the landscapes we will break the problem into cases determined by the number of vertices in the dual graphs of each landscape.  We first consider landscapes whose dual graphs contain two, three, and four vertices.  Later it will be shown that landscapes containing more than four vertices are not valid, but proof of this will be deferred until some necessary preliminaries are established.  Also note that due to the symmetries of $\mathcal P_8$ and the presence of a formula for switching from one shared-face to another, when considering points on adjacent faces in Theorem \ref{octahedron_Adjacent_Face}, we can assume without loss of generality that the representations of the points $p_1$ and $p_2$ are such that $p_1=(F_{n_1}, F_{n_2}, x_1, y_1)$ and $p_2=(F_{n_2}, F_{n_1}, x_2, y_2)$.

\begin{theorem}\label{octahedron_Adjacent_Face}
Let $F_{n_1}$ and $F_{n_2}$ be distinct and adjacent faces of $\mathcal P_8$, $p_1 = (F_{n_1}, F_{n_2}, x_1, y_1)$ and  $p_2 = (F_{n_2}, F_{n_1}, x_2, y_2)$.  Then there is one landscape $L_1(F_{n_1} \rightarrow F_{n_2})$ of $\mathcal P_8$ whose dual-graph is a path of two vertices and \[\abs{T_{L_1}(p_1, p_2)}=\sqrt{(x_1+x_2-1)^2+(y_1+y_2)^2}.\]
\end{theorem}

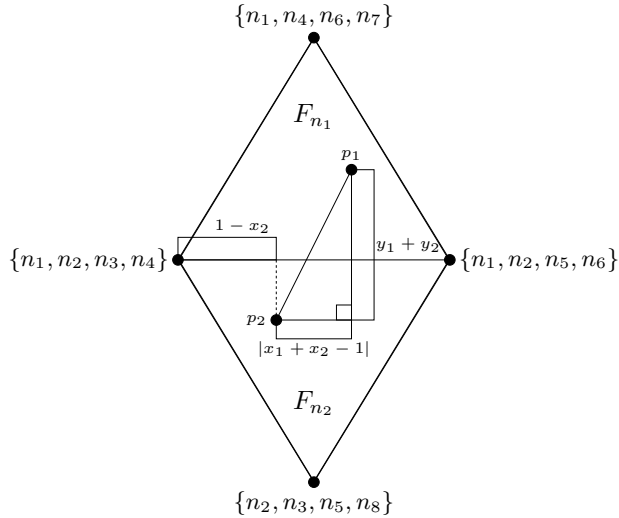
\begin{figure}[h!]
  \centering
\begin{tikzpicture}[scale=1]
    \node[diamond,
    draw,
    minimum width =3.6cm,
    minimum height =5.9cm] (d) at (3,4) {};
    \filldraw[fill=black,draw=black] (d.north) circle (2pt);
    \filldraw[fill=black,draw=black] (d.west) circle (2pt);
    \filldraw[fill=black,draw=black] (d.south) circle (2pt);
    \filldraw[fill=black,draw=black] (d.east) circle (2pt);
    \draw[line width = 0.05mm] (d.north) -- (d.west) -- (d.south) -- (d.east) -- cycle;
    \draw[line width=0.05mm] (d.east) -- (d.west);
    \node[above][font = {\small}] at (d.north) {$\{n_{1},n_{4},n_{6},n_{7}\}$};
    \node[below][font = {\small}] at (d.south) {$\{n_{2},n_{3},n_{5},n_{8}\}$};
    \node[left][font = {\small}] at (d.west) {$\{n_{1},n_{2},n_{3},n_{4}\}$};
    \node[right][font = {\small}] at (d.east) {$\{n_{1},n_{2},n_{5},n_{6}\}$};
    \node[above, inner sep=25] at (d.south) {$F_{n_2}$};
    \node[below, inner sep=25] at (d.north) {$F_{n_1}$};
    \draw [line width=0.01mm] (d.west) rectangle (2.5,4.3);
    \draw[ultra thin,dash pattern={on 1pt}] (2.5,4) -- (2.5,3.2);
    \filldraw[fill=black,draw=black] (2.5,3.2) circle (2pt);
    \node[left][font = {\tiny}] at (2.5,3.2) {$p_{2}$};
    \draw [ultra thin] (2.5,3.2) rectangle (3.5,2.95);
    \node[below, inner sep=.5][font = {\tiny}] at (3,2.95) {$|x_{1}+x_{2}-1|$};
    \draw [ultra thin] (3.5,3.2) rectangle (3.8,5.2);
    \filldraw[fill=black,draw=black] (3.5,5.2) circle (2pt);
    \node [right, inner sep=.5] [font = {\tiny}] at (3.8,4.2) {$y_{1}+y_{2}$};
    \draw [ultra thin] (3.5,5.2) -- (2.5,3.2);
    \draw [ultra thin] (3.3,3.2) rectangle (3.5,3.4);
    \node[above] [font = {\tiny}] at (3.5,5.2) {$p_{1}$};
    \node[font = {\tiny}] at (2.05,4.45) {$1-x_{2}$};
\end{tikzpicture}
    \caption{$L_1(F_{n_1} \rightarrow F_{n_2})$ and $T_{O_1}(p_1, p_2)$}
    \label{trail 1}
\end{figure}

\setcounter{case}{0}

\begin{proof}
Suppose $\{n_1, n_2, n_3, n_4, n_5, n_6, n_7,n_8\}=\{1, 2, 3, 4, 5, 6, 7, 8\}$ such that in nets of our copy of $\mathcal P_8$ the faces 
$F_{n_2}$, $F_{n_4}$, and $F_{n_6}$ occur in the same clockwise order around the face $F_{n_1}$ as the faces $F_2$, $F_4$, and $F_6$ occur around the face $F_1$, and $F_{n_b}$ and $F_{n_c}$ are opposite faces given that $b+c=9$.  Once we have constructed landscape $L_1$ we will then determine $p_2(L_1, F_{n_1},F_{n_2})$ so that we can use it to develop formulas for $\abs{T_{L_1}(p_1,p_2)}$ and $T_{O_1}(p_1,p_2)$, where $O_1=(L_1,F_{n_1},F_{n_2})$.

Since we are constructing landscapes from $F_{n_1}$ to $F_{n_2}$, and $F_{n_1}$ and $F_{n_2}$ have only one edge in common, there is only one landscape which allows a trail to start in $F_{n_1}$, pass directly from $F_{n_1}$ to $F_{n_2}$, and end in $F_{n_2}$. This landscape is given in Figure \ref{trail 1}, which we will refer to as $L_1$.  Since $p_2$'s home-face is $F_{n_2}$ and $p_2$'s shared-face is $F_{n_1}$, in the orientation $(L_1, F_{n_1}, F_{n_2})$ the point $p_2$ has been rotated $180$ degrees about the origin and shifted $1$ unit to the right from the standard position with respect to the representation $(F_{n_2},F_{n_1},x_2,y_2)$. Due to this $p_2(L_1, F_{n_1}, F_{n_2})=(1-x_2, -y_2)$.  Thus we get that
\[\abs{T_{L_1}(p_1, p_2)}=\sqrt{(x_1+x_2-1)^2+(y_1+y_2)^2}.\]

\end{proof}

Again, due to the symmetries of $\mathcal P_8$ and the presence of a formula for switching from one shared-face to another, when considering points on faces which are neither adjacent nor opposite in Theorem \ref{octahedron_non}, we can assume without loss of generality that the representations of the points $p_1$ and $p_2$ are such that $p_1=(F_{n_1}, F_{n_2}, x_1, y_1)$ and $p_2=(F_{n_5}, F_{n_6}, x_2, y_2)$.

\setcounter{case}{0}

\begin{theorem}\label{octahedron_non}
Let $F_{n_1}$ and $F_{n_5}$ be distinct faces of $\mathcal P_8$ which are neither opposite nor adjacent, $p_1 = (F_{n_1}, F_{n_2}, x_1, y_1)$ and  $p_2 = (F_{n_5}, F_{n_6}, x_2, y_2)$.  Then there are two landscapes whose dual graph is a path of three vertices, $L_2(F_{n_1}\rightarrow F_{n_5})$ and $L_3(F_{n_1}\rightarrow F_{n_5})$, with $\abs{T_{L_2}(p_1, p_2)}$ and $\abs{T_{L_3}(p_1, p_2)}$ given below:
\[\abs{T_{L_2}(p_1, p_2)} = \sqrt{\left(\frac {1-{x_1}+\sqrt{3}{y_1}}{2}-x_2+1\right)^2+\left(\frac{\sqrt{3}-\sqrt{3}{x_1}-{y_1}}{2}-y_2\right)^2},\] \[\abs{T_{L_3}(p_1, p_2)} = \sqrt{\left(x_1-1-\frac {1-{x_2}+\sqrt{3}{y_2}}{2}\right)^2+\left(y_1-\frac{\sqrt{3}-\sqrt{3}{x_2}-{y_2}}{2}\right)^2}.\]
\end{theorem}
\begin{proof}
Let $\{n_1, n_2, n_3, n_4, n_5, n_6, n_7,n_8\}=\{1, 2, 3, 4, 5, 6, 7, 8\}$ such that in nets of our copy of $\mathcal P_8$ the faces $F_{n_2}$, $F_{n_4}$, and $F_{n_6}$ occur in the same clockwise order around the face $F_{n_1}$ as the faces $F_2$, $F_4$, and $F_6$ occur around the face $F_1$, and $F_{n_b}$ and $F_{n_c}$ are opposite faces given that $b+c=9$.  For the ease of notation in diagrams, we will let $u = p_1(L_2,F_{n_1},F_{n_6})_x$, $v =p_1(L_2,F_{n_1},F_{n_6})_y, s = p_2(L_3,F_{n_5},F_{n_2})_x$, and $t =p_2(L_3,F_{n_5},F_{n_2})_y$.

Since we are constructing landscapes of the form $L_a(F_{n_1}\rightarrow F_{n_5})$, the first face must be $F_{n_1}$, the third face must be $F_{n_5}$, and the second face must be a face which is adjacent to $F_{n_1}$. In particular, the second face must be a member of the set $\{F_{n_2}, F_{n_6}\}$ to which $F_{n_5}$ is also adjacent. Since there will be two choices for the second face, there will be a total of two such landscapes, depicted in Figures \ref{Trail 2} and \ref{Trail 3}.

\begin{figure}[h!]
  \centering
\begin{tikzpicture}[scale=1]
    \node[isosceles triangle,
    isosceles triangle apex angle=60,
    rotate=270,
    draw,
    minimum size=3cm] (T)at (0,0) {};
    \node[above][font = {\small}] at (T.right corner) {$\{n_{2},n_{3},n_{5},n_{8}\}$};
        \node[above][font = {\small}] at (T.left corner) {$\{n_{1},n_{2},n_{3},n_{4}\}$};
    \node[font = {\small}] at (1.3,-1.8) {$\{n_{1},n_{2},n_{5},n_{6}\}$};
    \filldraw[fill=black,draw=black] (T.right corner) circle (2pt);
    \filldraw[fill=black,draw=black] (T.left corner) circle (2pt);
    \filldraw[fill=black,draw=black] (T.apex) circle (2pt);
    \node[above, inner sep=60] at (T.apex) {$F_{n_2}$};
    \node[below, inner sep=20] at (T.left corner) {$F_{n_1}$};
    \node[below, inner sep=20] at (T.right corner) {$F_{n_5}$};
    \filldraw[fill=black,draw=black] (3.5,-2) circle (2pt);
    \filldraw[fill=black,draw=black] (-3.5,-2) circle (2pt);
    \draw[line width=0.05mm] (1.75,1) -- (3.5,-2) -- (0,-2);
    \draw[line width=0.05mm] (-1.75,1) -- (-3.5,-2) -- (0,-2);
    \node[left][font = {\small}] at (-3.5,-2) {$\{n_{5},n_{6},n_{7},n_{8}\}$};
    \node[right][font = {\small}] at (3.5,-2) {$\{n_{1},n_{4},n_{6},n_{7}\}$};
    \filldraw[fill=black,draw=black] (2.2,-.7) circle (2pt);
    \node[below][font = {\small}] at (2.2,-.7) {$p_{1}$};
    \filldraw[fill=black,draw=black] (-1.2,-1.4) circle (2pt);
    \node[below][font = {\small}] at (-1.2,-1.4) {$p_{2}$};    
    \draw[ultra thin] (-1.2,-1.4) -- (2.2,-.7);
    \draw [ultra thin] (-1.2,-1.4) rectangle (-1.5,-.7);
    \draw [ultra thin] (2.2,-.7) rectangle (-1.2,-.4);
    \draw [ultra thin] (-1.2,-.7) rectangle (-1,-.9);
    \node[above, inner sep=2][font = {\small}] at (.1,-.4) {$1+u-x_{2}$};
    \node[left, inner sep=2][font = {\small}] at (-1.5,-.95) {$|v-y_{2}|$};
    \draw [ultra thin] (-3.5,-2) rectangle (2.2,-2.4);
    \node[below][font = {\small}] at (-.65,-2.4) {$1+u$};  
\end{tikzpicture}
    \caption{$L_2(F_{n_1} \rightarrow F_{n_5})$ and $T_{O}(p_1, p_2),$} where $O=(L_2,F_{n_1},F_{n_6})$
  \label{Trail 2}
\end{figure}
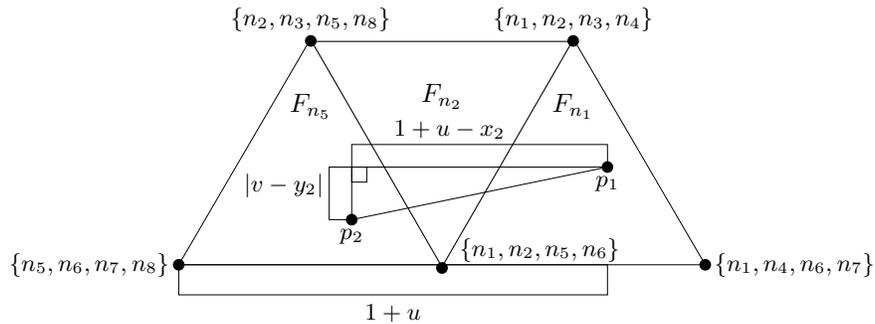

We will first consider trails on $L_2(F_{n_1} \rightarrow F_{n_5})$ depicted in Figure \ref{Trail 2}.  Since $F_{n_2}$ is the shared face of $p_1$, by application of Lemma \ref{octahedronrepresentation}, $p_1=\left(F_{n_1}, F_{n_6}, \frac {1-{x_1}+\sqrt{3}{y_1}}{2}, \frac{\sqrt{3}-\sqrt{3}{x_1}-{y_1}}{2}\right)$. Since $p_2$'s home-face is $F_{n_5}$ and $p_2$'s shared-face is $F_{n_6}$, in the orientation $(L_2, F_{n_1}, F_{n_6})$ the point $p_2$ has been shifted $1$ unit to the left from $p_2$'s standard position with respect to the representation $(F_{n_5},F_{n_6},x_2,y_2)$.  Due to this $p_2(L_2, F_{n_1}, F_{n_6})=(x_2-1, y_2)$.  Thus we get that
\[\abs{T_{L_2}(p_1, p_2)} = \sqrt{\left(\frac {1-{x_1}+\sqrt{3}{y_1}}{2}-x_2+1\right)^2+\left(\frac{\sqrt{3}-\sqrt{3}{x_1}-{y_1}}{2}-y_2\right)^2}.\]

\begin{figure}[h!]
  \centering
\begin{tikzpicture}[scale=1]
    \node[isosceles triangle,
    isosceles triangle apex angle=60,
    rotate=270,
    draw,
    minimum size=3cm] (T)at (0,0) {};
    \node[above][font = {\small}] at (T.right corner) {$\{n_{1},n_{4},n_{6},n_{7}\}$};
    \node[above][font = {\small}] at (T.left corner) {$\{n_{5},n_{6},n_{7},n_{8}\}$};
    \node[font = {\small}] at (-1.3,-1.825) {$\{n_{1},n_{2},n_{5},n_{6}\}$};
    \filldraw[fill=black,draw=black] (T.right corner) circle (2pt);
    \filldraw[fill=black,draw=black] (T.left corner) circle (2pt);
    \filldraw[fill=black,draw=black] (T.apex) circle (2pt);
    \node[above, inner sep=50] at (T.apex) {$F_{n_6}$};
    \node[below, inner sep=30] at (T.left corner) {$F_{n_5}$};
    \node[below, inner sep=30] at (T.right corner) {$F_{n_1}$};
    \filldraw[fill=black,draw=black] (3.5,-2) circle (2pt);
    \filldraw[fill=black,draw=black] (-3.5,-2) circle (2pt);
    \draw[line width=0.05mm] (1.75,1) -- (3.5,-2) -- (0,-2);
    \draw[line width=0.05mm] (-1.75,1) -- (-3.5,-2) -- (0,-2);
    \node[left][font = {\small}] at (-3.5,-2) {$\{n_{1},n_{2},n_{3},n_{4}\}$};
    \node[right][font = {\small}] at (3.5,-2) {$\{n_{2},n_{3},n_{5},n_{8}\}$};
    \filldraw[fill=black,draw=black] (-1.2,-.7) circle (2pt);
    \node[above][font = {\small}] at (-1.2,-.7) {$p_{1}$};
    \filldraw[fill=black,draw=black] (2.2,-1.4) circle (2pt);
    \node[right][font = {\small}] at (2.2,-1.4) {$p_{2}$};    
    \draw[ultra thin] (2.2,-1.4) -- (-1.2,-.7);
    \draw [ultra thin] (-1.2,-1.4) rectangle (-1.5,-.7);
    \draw [ultra thin] (2.2,-1.4) rectangle (-1.2,-1.65);
    \draw [ultra thin] (-1.2,-1.4) rectangle (-1,-1.2);
    \node[below, inner sep=2][font = {\small}] at (1.2,-1.625) {$1+s-x_{1}$};
    \node[left, inner sep=2][font = {\small}] at (-1.5,-.95) {$|y_{1}-t|$};
    \draw [ultra thin] (-3.5,-2) rectangle (2.2,-2.4);
    \node[below][font = {\small}] at (-.65,-2.4) {$1+s$};  
    \end{tikzpicture}
     \caption{$L_3(F_{n_1} \rightarrow F_{n_5})$ and $T_{O}(p_1, p_2),$} where $O=(L_3,F_{n_1},F_{n_5})$
  \label{Trail 3}
\end{figure}
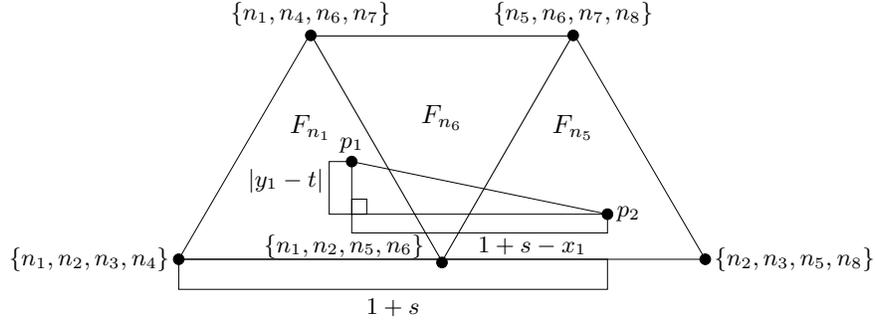

We now consider trails on $L_3(F_{n_1} \rightarrow F_{n_5})$ depicted in Figure \ref{Trail 3}.  Since $F_{n_6}$ is the shared face of $p_2$, by application of Lemma \ref{octahedronrepresentation}, $p_2=\left(F_{n_5}, F_{n_2}, \frac {1-{x_2}+\sqrt{3}{y_2}}{2}, \frac{\sqrt{3}-\sqrt{3}{x_2}-{y_2}}{2}\right)$. Since $p_1$'s home-face is $F_{n_1}$ and $p_1$'s shared-face is $F_{n_2}$, in the orientation $(L_3, F_{n_5}, F_{n_2})$ the point $p_1$ has been shifted $1$ unit to the left from $p_1$'s standard position with respect to the representation $(F_{n_1},F_{n_2},x_1,y_1)$.  Due to this $p_1(L_3, F_{n_5}, F_{n_2})=(x_1-1, y_1)$.  Thus we get that
\[\abs{T_{L_3}(p_1, p_2)} = \sqrt{\left(x_1-1-\frac {1-{x_2}+\sqrt{3}{y_2}}{2}\right)^2+\left(y_1-\frac{\sqrt{3}-\sqrt{3}{x_2}-{y_2}}{2}\right)^2}.\]
\end{proof}

As before, due to the symmetries of $\mathcal P_8$ and the presence of a formula for switching from one shared-face to another, when considering points on opposite faces in Theorem \ref{octahedron_opposite_Face}, we can assume without loss of generality that the representations of the points $p_1$ and $p_2$ are such that $p_1=(F_{n_1}, F_{n_2}, x_1, y_1)$ and $p_2=(F_{n_8}, F_{n_7}, x_2, y_2)$.

\begin{theorem}\label{octahedron_opposite_Face}
Let $F_{n_1}$ and $F_{n_8}$ be distinct and opposite faces of $\mathcal P_8$, $p_1=(F_{n_1},F_{n_2},x_1,y_1)$ and $p_2=(F_{n_8},F_{n_7},x_2,y_2)$.  Then there are at least 6 landscapes $L_a(F_{n_1} \rightarrow F_{n_8})$ which have a dual-graph path of four vertices. For each $a \in \mathbb N$ with $4 \leq a \leq 9$, the trail length $\abs{T_{L_a}(p_1,p_2)}$ is given in the table below:

\begingroup
\begin{center}
\begin{tabular}{|c|c|}
    \hline
    Landscape & Trail Length Formula\\
    \hline
    $L_4$ &\rule{-4pt}{20pt} $\sqrt{\left(\frac{2-x_1-\sqrt{3}y_1}{2}-\frac{-2+x_2+\sqrt{3}y_2}{2}\right)^2+\left(\frac{\sqrt{3}x_1-y_1}{2}-\frac{2\sqrt{3}-\sqrt{3}x_2+y_2}{2}\right)^2}$\\[9pt]
    \hline
    $L_5$ &\rule{-4pt}{20pt} $\sqrt{\left(\frac{1-x_1+\sqrt{3}y_1}{2}+x_2\right)^2+\left(\frac{\sqrt{3}-\sqrt{3}x_1-y_1}{2}+y_2-\sqrt{3}\right)^2}$\\[9pt]
    \hline
    $L_6$ &\rule{-4pt}{20pt} $\sqrt{\left(x_1-\frac{-1+x_2-\sqrt{3}y_2}{2}\right)^2+\left(y_1-\frac{\sqrt{3}x_2+y_2+\sqrt{3}}{2}\right)^2}$\\[9pt]
    \hline
    $L_7$ &\rule{-4pt}{20pt} $\sqrt{\left(x_1-\frac{2+x_2+\sqrt{3}y_2}{2}\right)^2+\left(y_1-\frac{-\sqrt{3}x_2+y_2+2\sqrt{3}}{2} \right)^2}$\\[9pt]
    \hline
    $L_8$ &\rule{-4pt}{20pt} $\sqrt{\left(\frac{2-x_1-\sqrt{3}y_1}{2}+x_2-2\right)^2+\left(\frac{\sqrt{3}x_1-y_1}{2}+y_2-\sqrt{3}\right)^2}$\\[9pt]
    \hline
    $L_9$ &\rule{-4pt}{20pt} $\sqrt{\left(\frac {1-{x_1}+\sqrt{3}{y_1}}{2}-\frac {3+{x_2}-\sqrt{3}{y_2}}{2}\right)^2+\left(\frac{\sqrt{3}-\sqrt{3}{x_1}-{y_1}}{2}-\frac{\sqrt{3}+\sqrt{3}{x_2}+{y_2}}{2}\right)^2}$\\[9pt]
    \hline
\end{tabular}
\end{center}
\endgroup
\end{theorem}

\setcounter{case}{0}
\begin{proof}
Let $\{n_1, n_2, n_3, n_4, n_5, n_6, n_7,n_8\}=\{1, 2, 3, 4, 5, 6, 7, 8\}$ such that in nets of our copy of $\mathcal P_8$ the faces $F_{n_2}$, $F_{n_4}$, and $F_{n_6}$ occur in the same clockwise order around the face $F_{n_1}$ as the faces $F_2$, $F_4$, and $F_6$ occur around the face $F_1$, and $F_{n_b}$ and $F_{n_c}$ are opposite faces given that $b+c=9$.  For the ease of notation in diagrams and arguments, for each $a$ with $4 \leq a \leq 9$, we will let $u_a = p_1(L_a,F_{n_1},F_{n_{\ell}})_x$, $v_a =p_1(L_a,F_{n_1},F_{n_{\ell}})_y, s_a = p_2(L_a,F_{n_6},F_{n_m})_x$, and $t_a =p_2(L_a,F_{n_6},F_{n_m})_y$.

Since we are constructing landscapes of the form $L_a(F_{n_1}\rightarrow F_{n_8})$, the first face must be $F_{n_1}$, the fourth face must be $F_{n_8}$, the second face, $F_{n_i}$ must be a face which is adjacent to $F_{n_1}$, and the third face, $F_{n_j}$, must be a face which is adjacent to the second face and $F_{n_8}$. In particular, $F_{n_i}$ must be a member of the set $\{F_{n_2}, F_{n_4}, F_{n_6}\}$ and $F_{n_j}$ will be a member of the set $\{F_{n_3},F_{n_5},F_{n_7}\}\setminus\{F_{n_{9-i}}\}$. Since there are three choices for $F_{n_i}$ and for each choice of $F_{n_i}$ there will in turn be two choices for the third face $F_{n_j}$, there will be a total of six such landscapes. These landscapes will produce two general structures, depicted in Figures \ref{Trail 4} and \ref{Trail 5}.

\begin{figure}[h]
  \centering
\begin{tikzpicture}[scale=1]
\filldraw[fill=black,draw=black] (0,6) circle (2pt);
\filldraw[fill=black,draw=black] (0,0) circle (2pt);
\filldraw[fill=black,draw=black] (3.5,0) circle (2pt);
\filldraw[fill=black,draw=black] (-3.5,6) circle (2pt);
\filldraw[fill=black,draw=black] (1.75,3) circle (2pt);
\filldraw[fill=black,draw=black] (-1.75,3) circle (2pt);
\draw[line width = 0.05mm] (0,0) -- (3.5,0) -- (1.75,3) -- (0,0) -- (-1.75,3) -- (1.75,3) -- (0,6) -- (-1.75,3) -- (-3.5,6) -- (0,6);
\node[below][font = {\small}] at (3.5,0) {$\{n_{1},n_{k},n_{\ell},n_{m}\}$};
\node[above][font = {\small}] at (0,6) {$\{n_{j},n_{k},n_{m},n_{8}\}$};
\node[above][font = {\small}] at (-3.5,6) {$\{n_{\ell},n_{m},n_{q},n_{8}\}$};
\node[below][font = {\small}] at (0,0) {$\{n_{1},n_{i},n_{\ell},n_{q}\}$};
\node[right][font = {\small}] at (1.75,3) {$\{n_{1},n_{i},n_{j},n_{k}\}$};
\node[font = {\small}] at (-.8,2.8) {$\{n_{i},n_{j},n_{q},n_{8}\}$};
\node[font = ] at (1.75,2.3) {$F_{n_1}$};
\node at (-.7,5.65) {$F_{n_8}$};
\node at (1.15,2.75) {$F_{n_i}$};
\node at (1.15,3.4) {$F_{n_j}$};
\draw[ultra thin,dash pattern={on 1pt}] (0,6) -- (0,0);
\filldraw[fill=black,draw=black] (2.3,1.15) circle (2pt);
\filldraw[fill=black,draw=black] (-2.3,5.3) circle (2pt);
\node[right][font = {\small}] at (2.3,1.15) {$p_{1}$};
\node[above][font = {\small}] at (-2.3,5.3) {$p_{2}$};
\draw[line width = 0.05mm] (2.3,1.15) -- (-2.3,5.3);
\draw[line width=0.05mm](-2.3,5.3) rectangle (-2.6,1.15);
\draw[line width=0.05mm](2.3,1.15) rectangle (-2.3,.85);
\draw[line width=0.05mm](-2.3,1.15) rectangle (-2.1,1.35);
\draw[ultra thin,dash pattern={on 1pt}] (-1.75,3) -- (-2.3,3);
\draw[ultra thin] (-2.3,3) -- (-2.6,3);
\node[left][font = {\small}] at (-2.6,4.15) {$\frac{\sqrt{3}}{2}-t_a$};
\node[left][font = {\small}] at (-2.6,2.1) {$\frac{\sqrt{3}}{2}-v_a$};
\node[below][font = {\small}] at (-1.5,.85) {$u_a+s_a$};
\end{tikzpicture}
    \caption{
  $L_a(F_{n_1} \rightarrow F_{n_8})$ and $T_{O}(p_1, p_2)$,\\
  where $O=(L_a,F_{n_1},F_{n_{\ell}})$ and $a \in \{4,5,6\}$
}
  \label{Trail 4} 
\end{figure}
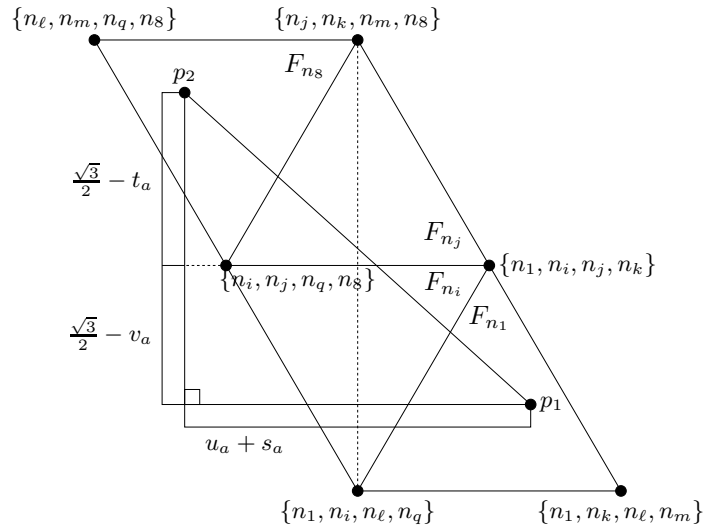

We will first construct landscapes having the general structure depicted in Figure \ref{Trail 4}, namely $L_a(F_{n_1} \rightarrow F_{n_8})$, where $a \in \{4,5,6\}$. Since each of the three landscapes will have the same general structure, if we apply Lemma \ref{octahedronrepresentation} a sufficient number of times so that the face $F_{n_{\ell}}$ is the shared face for the representation of $p_1$ and $F_{n_m}$ is the shared face for the representation of $p_2$, then the same transformation takes us from $p_2$'s standard position with respect to this representation to $p_2$'s new position in the orientation $(L_a, F_{n_1}, F_{n_{\ell}})$. Specifically, in each case, since $p_2$'s home-face is $F_{n_8}$ and shared-face in this representation is $F_{n_m}$, in the orientation $(L_a, F_{n_1}, F_{n_{\ell}})$, the point $p_2$ has been 
rotated $180$ degrees counterclockwise about the origin and shifted $\sqrt{3}$ units up from $p_2$'s standard position with respect to the representation $(F_{n_8},F_{n_m},u_a,v_a)$, and thus $p_2(L_a,F_{n_1},F_{n_{\ell}})=(-u_a,-v_a+\sqrt{3})$.

For $L_4$, we will respectively let the second and third faces, $F_{n_i}$ and $F_{n_j}$, be $F_{n_6}$ and $F_{n_5}$. In turn, we get that $n_k=n_2$, $n_{\ell}=n_4$, $n_m=n_3$, and $n_q=n_7$. By application of Lemma \ref{octahedronrepresentation} twice, we see that $p_1= \left(F_{n_1}, F_{n_4}, \frac{2-x_1-\sqrt{3}y_1}{2},\frac{\sqrt{3}x_1-y_1}{2}\right)$ and $p_2= \left(F_{n_8}, F_{n_3},\frac{2-x_2-\sqrt{3}y_2}{2},\frac{\sqrt{3}x_2-y_2}{2}\right)$. Thus, after application of the transformation  $p_2(L_4,F_{n_1},F_{n_4}) = \left(\frac{-2+x_2+\sqrt{3}y_2}{2},\frac{2\sqrt{3}-\sqrt{3}x_2+y_2}{2}\right)$ and we get that 
\[\abs{T_{L_4}(p_1, p_2)} = \sqrt{\left(\frac{2-x_1-\sqrt{3}y_1}{2}-\frac{-2+x_2+\sqrt{3}y_2}{2}\right)^2+\left(\frac{\sqrt{3}x_1-y_1}{2}-\frac{2\sqrt{3}-\sqrt{3}x_2+y_2}{2}\right)^2}.\]

For $L_5$, we will respectively let the second and third faces, $F_{n_i}$ and $F_{n_j}$, be $F_{n_2}$ and $F_{n_3}$. In turn, we get that $n_k=n_4$, $n_{\ell}=n_6$, $n_q=n_5$, and $n_m=n_7$. By an application of Lemma \ref{octahedronrepresentation} to $p_1$, we see that $p_1= \left(F_{n_1}, F_{n_6}, \frac {1-{x_1}+\sqrt{3}{y_1}}{2}, \frac{\sqrt{3}-\sqrt{3}{x_1}-{y_1}}{2}\right)$. Thus, after application of the transformation  
$p_2(L_5,F_{n_1},F_{n_6}) = (-x_2,-y_2+\sqrt{3})$ and we get that 
\[\abs{T_{L_5}(p_1, p_2)} = \sqrt{\left(\frac{1-x_1+\sqrt{3}y_1}{2}+x_2\right)^2+\left(\frac{\sqrt{3}-\sqrt{3}x_1-y_1}{2}+y_2-\sqrt{3}\right)^2}.\]

For $L_{6}$, we will respectively let the second and third faces, $F_{n_i}$ and $F_{n_j}$, be $F_{n_4}$ and $F_{n_7}$. In turn, we get that $n_{\ell}=n_2$, $n_k=n_6$, $n_q=n_3$, and $n_m=n_5$. By an application of Lemma \ref{octahedronrepresentation} to $p_2$, we see that $p_2= \left(F_{n_8}, F_{n_7}, \frac {1-{x_2}+\sqrt{3}{y_2}}{2}, \frac{\sqrt{3}-\sqrt{3}{x_2}-{y_2}}{2}\right)$. Thus, after application of the transformation  
$p_2(L_{6},F_{n_1},F_{n_2}) = \left(\frac{-1+x_2-\sqrt{3}y_2}{2},\frac{\sqrt{3}x_2+y_2+\sqrt{3}}{2}\right)$ and we get that 

\[\abs{T_{L_{6}}(p_1, p_2)} = \sqrt{\left(x_1-\frac{-1+x_2-\sqrt{3}y_2}{2}\right)^2+\left(y_1-\frac{\sqrt{3}x_2+y_2+\sqrt{3}}{2}\right)^2}.\]

We will now construct landscapes having the general structure depicted in Figure \ref{Trail 5}, namely $L_a(F_{n_1} \rightarrow F_{n_8})$, where $a \in \{7,8,9\}$.  Since each of the three landscapes will have the same general structure, if we apply Lemma \ref{octahedronrepresentation} a sufficient number of times so that the face $F_{n_{\ell}}$ is the shared face for the representation of $p_1$ and $F_{n_m}$ is the shared face for the representation of $p_2$, then the same transformation takes us from $p_2$'s standard position with respect to this representation to $p_2$'s new position in the orientation $(L_a, F_{n_1}, F_{n_{\ell}})$. Specifically, in each case, since $p_2$'s home-face is $F_{n_8}$ and shared-face in this representation is $F_{n_m}$, in the orientation $(L_a, F_{n_1}, F_{n_{\ell}})$, the point $p_2$ has been 
rotated $180$ degrees counterclockwise about the origin and shifted $2$ units to the right and $\sqrt{3}$ units up from $p_2$'s standard position with respect to the representation $(F_{n_8},F_{n_m},u_a,v_a)$ and thus $p_2(L_a,F_{n_1},F_{n_{\ell}})=(-u_a+2,-v_a+\sqrt{3})$.

For $L_{7}$, we will respectively let the second and third faces, $F_{n_i}$ and $F_{n_j}$, be $F_{n_6}$ and $F_{n_7}$. In turn, we get that $n_{\ell}=n_2$, $n_k=n_4$, $n_m=n_3$, and $n_q=n_5$. By two applications of Lemma \ref{octahedronrepresentation} to $p_2$, we see that $p_2=\left(F_{n_8}, F_{n_3},\frac{2-x_2-\sqrt{3}y_2}{2},\frac{\sqrt{3}x_2-y_2}{2}\right)$. Thus, after application of the transformation $p_2(L_7,F_{n_1},F_{n_2})=\left(\frac{2+x_2+\sqrt{3}y_2}{2},\frac{-\sqrt{3}x_2+y_2+2\sqrt{3}}{2}\right)$ and we get that 
\[\abs{T_{L_{7}}(p_1, p_2)} = \sqrt{\left(x_1-\frac{2+x_2+\sqrt{3}y_2}{2}\right)^2+\left(y_1-\frac{-\sqrt{3}x_2+y_2+2\sqrt{3}}{2} \right)^2}.\]

For $L_{8}$, we will respectively let the second and third faces, $F_{n_i}$ and $F_{n_j}$, be $F_{n_2}$ and $F_{n_5}$. In turn, we get that $n_{\ell}=n_4$, $n_k=n_6$, $n_q=n_3$, and $n_m=n_7$. By two applications of Lemma \ref{octahedronrepresentation} to $p_1$, we see that $p_1= \left(F_{n_1}, F_{n_4}, \frac{2-x_1-\sqrt{3}y_1}{2},\frac{\sqrt{3}x_1-y_1}{2}\right)$. Thus, after application of the transformation  
$p_2(L_{8},F_{n_1},F_{n_4}) = (-x_2+2,-y_2+\sqrt{3})$ and we get that 
\[\abs{T_{L_{8}}(p_1, p_2)} = \sqrt{\left(\frac{2-x_1-\sqrt{3}y_1}{2}+x_2-2)\right)^2+\left(\frac{\sqrt{3}x_1-y_1}{2}+y_2-\sqrt{3}\right)^2}.\]

For $L_{9}$, we will respectively let the second and third faces, $F_{n_i}$ and $F_{n_j}$, be $F_{n_4}$ and $F_{n_3}$. In turn, we get that $n_k=n_2$, $n_{\ell}=n_6$, $n_q=n_7$, and $n_m=n_5$. By an application of Lemma \ref{octahedronrepresentation} to $p_1$ and $p_2$, we see that $p_1=\left(F_{n_1}, F_{n_6}, \frac {1-{x_1}+\sqrt{3}{y_1}}{2}, \frac{\sqrt{3}-\sqrt{3}{x_1}-{y_1}}{2}\right)$ and $p_2= \left(F_{n_8}, F_{n_7}, \frac {1-{x_2}+\sqrt{3}{y_2}}{2}, \frac{\sqrt{3}-\sqrt{3}{x_2}-{y_2}}{2}\right)$. Thus, after application of the transformation  
$p_2(L_{9},F_{n_1},F_{n_6}) = \left(\frac {3+{x_2}-\sqrt{3}{y_2}}{2}, \frac{\sqrt{3}+\sqrt{3}{x_2}+{y_2}}{2}\right)$ and we get that 

\[\abs{T_{L_{9}}(p_1, p_2)} = \sqrt{\left(\frac {1-{x_1}+\sqrt{3}{y_1}}{2}-\frac {3+{x_2}-\sqrt{3}{y_2}}{2}\right)^2+\left(\frac{\sqrt{3}-\sqrt{3}{x_1}-{y_1}}{2}-\frac{\sqrt{3}+\sqrt{3}{x_2}+{y_2}}{2}\right)^2}.\]

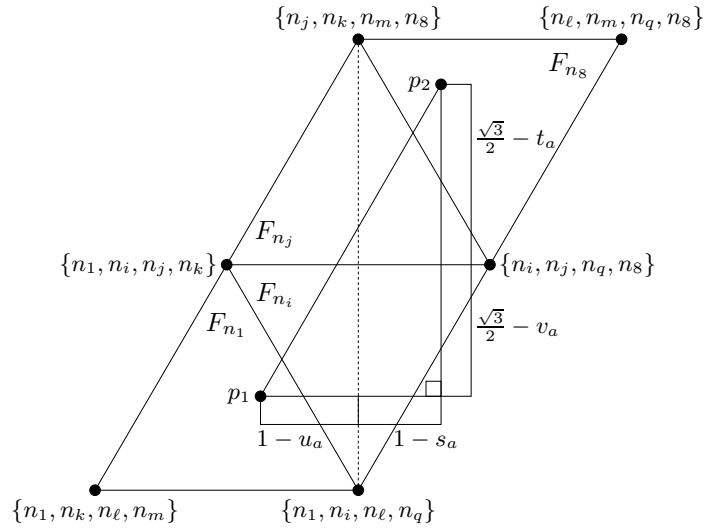
\begin{figure}[h]
  \centering
\begin{tikzpicture}[scale=1]
\filldraw[fill=black,draw=black] (0,6) circle (2pt);
\filldraw[fill=black,draw=black] (0,0) circle (2pt);
\filldraw[fill=black,draw=black] (-3.5,0) circle (2pt);
\filldraw[fill=black,draw=black] (3.5,6) circle (2pt);
\filldraw[fill=black,draw=black] (1.75,3) circle (2pt);
\filldraw[fill=black,draw=black] (-1.75,3) circle (2pt);
\draw[line width = 0.05mm] (0,0) -- (-3.5,0) -- (-1.75,3) -- (0,0) -- (1.75,3) -- (-1.75,3) -- (0,6) -- (1.75,3) -- (3.5,6) -- (0,6);
\node[below][font = {\small}] at (-3.5,0) {$\{n_{1},n_{k},n_{\ell},n_{m}\}$};
\node[above][font = {\small}] at (0,6) {$\{n_{j},n_{k},n_{m},n_{8}\}$};
\node[above][font = {\small}] at (3.5,6) {$\{n_{\ell},n_{m},n_{q},n_{8}\}$};
\node[below][font = {\small}] at (0,0) {$\{n_{1},n_{i},n_{\ell},n_{q}\}$};
\node[right][font = {\small}] at (1.75,3) {$\{n_{i},n_{j},n_{q},n_{8}\}$};
\node[left][font = {\small}] at (-1.75,3) {$\{n_{1},n_{i},n_{j},n_{k}\}$};
\node at (-1.75,2.2) {$F_{n_1}$};
\node at (2.8,5.65) {$F_{n_8}$};
\node at (-1.1,2.6) {$F_{n_i}$};
\node at (-1.1,3.4) {$F_{n_j}$};
\draw[ultra thin,dash pattern={on 1pt}] (0,6) -- (0,0);
\filldraw[fill=black,draw=black] (1.1,5.4) circle (2pt);
\node[left][font = {\small}] at (1.1,5.4) {$p_{2}$};
\filldraw[fill=black,draw=black] (-1.3,1.25) circle (2pt);
\node[left][font = {\small}] at (-1.3,1.25) {$p_{1}$};
\draw[line width = 0.05mm] (1.1,5.4) -- (-1.3,1.25);
\draw[line width=0.05mm](1.1,5.4) rectangle (1.5,1.25);
\draw[line width=0.05mm](-1.3,1.25) rectangle (1.1,.875);
\draw[line width=0.05mm](1.1,1.25) rectangle (.9,1.45);
\node[right, inner sep=1][font = {\small}] at (1.5,4.7) {$\frac{\sqrt{3}}{2}-t_a$};
\node[right, inner sep=1][font = {\small}] at (1.5,2.2) {$\frac{\sqrt{3}}{2}-v_a$};
\node[below, inner sep=2][font = {\small}] at (-.9,.875) {$1-u_a$};
\node[below, inner sep=2][font = {\small}] at (.9,.875) {$1-s_a$};
\draw[ultra thin](0,1.25) -- (0,.875);
\end{tikzpicture}
  \caption{
  $L_a(F_{n_1} \rightarrow F_{n_8})$ and $T_{O}(p_1, p_2)$,\\
  where $O=(L_a,F_{n_1},F_{n_{\ell}})$ and $a \in \{7,8,9\}$
}
  \label{Trail 5} 
\end{figure}
\end{proof}

\begin{theorem}\label{finaltheorem}
Let $L(F_{n_1} \rightarrow F_{n_i})$ with $i \in \{1,2,3,4,5,6,7,8\}$, be a landscape of $\mathcal P_8$.  If $L \not \in \{L_a\}_{a=1}^{9}$, then $L$ is not valid.
\end{theorem}

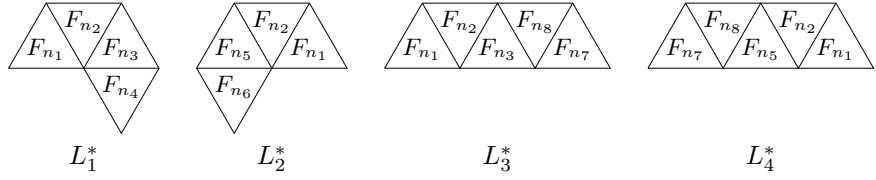
\begin{figure}[h]
  \centering
\begin{tikzpicture}[scale=1]

\draw [line width=0.05mm] (0,0) -- (.5,.866) -- (1.5,.866) -- (2,0) -- (1.5,-.866) -- (1,0) -- (0,0);
\draw [line width=0.05mm] (.5,.866) -- (1,0);
\draw [line width=0.05mm] (1.5,.866) -- (1,0);
\draw [line width=0.05mm] (2,0) -- (1,0);
0 \node[below, outer sep=10][font = {\small}] at (.5,.866) {$F_{n_1}$};
\node[above, outer sep=10][font = {\small}] at (1,0) {$F_{n_2}$};
\node[below, outer sep=10][font = {\small}] at (1.5,.866) {$F_{n_3}$};
\node[above, outer sep=10][font = {\small}] at (1.5,-.866) {$F_{n_4}$};

\draw [line width=0.05mm] (2.5,0) -- (3,.866) -- (4,.866) -- (4.5,0) -- (3.5,0) -- (3,-.866) -- (2.5,0);
\draw [line width=0.05mm] (2.5,0) -- (3.5,0);
\draw [line width=0.05mm] (4,.866) -- (3.5,0);
\draw [line width=0.05mm] (3,.866) -- (3.5,0);
\node[below, outer sep=10][font = {\small}] at (4,.866) {$F_{n_1}$};
\node[above, outer sep=10][font = {\small}] at (3.5,0) {$F_{n_2}$};
\node[below, outer sep=10][font = {\small}] at (3,.866) {$F_{n_5}$};
\node[above, outer sep=10][font = {\small}] at (3,-.866) {$F_{n_6}$};

\draw [line width=0.05mm] (5,0) -- (5.5,.866) -- (6.5,.866) -- (7.5,.866) -- (8,0) -- (7,0) -- (6,0) -- (5,0);
\draw [line width=0.05mm] (5.5,.866) -- (6,0) -- (6.5,.866) -- (7,0) -- (7.5,.866);
\node[below, outer sep=10][font = {\small}] at (5.5,.866) {$F_{n_1}$};
\node[above, outer sep=10][font = {\small}] at (6,0) {$F_{n_2}$};
\node[below, outer sep=10][font = {\small}] at (6.5,.866) {$F_{n_3}$};
\node[above, outer sep=10][font = {\small}] at (7,0) {$F_{n_8}$};
\node[below, outer sep=10][font = {\small}] at (7.5,.866) {$F_{n_7}$};

\draw [line width=0.05mm] (8.5,0) -- (9,.866) -- (10,.866) -- (11,.866) -- (11.5,0) -- (10.5,0) -- (9.5,0) -- (8.5,0);
\draw [line width=0.05mm] (9,.866) -- (9.5,0) -- (10,.866) -- (10.5,0) -- (11,.866);
\node[below, outer sep=10][font = {\small}] at (9,.866) {$F_{n_7}$};
\node[above, outer sep=10][font = {\small}] at (9.5,0) {$F_{n_8}$};
\node[below, outer sep=10][font = {\small}] at (10,.866) {$F_{n_5}$};
\node[above, outer sep=10][font = {\small}] at (10.5,0) {$F_{n_2}$};
\node[below, outer sep=10][font = {\small}] at (11,.866) {$F_{n_1}$};

\node at (1,-1.2) {$L^{*}_{1}$};
\node at (3.5,-1.2) {$L^{*}_{2}$};
\node at (6.5,-1.2) {$L^{*}_{3}$};
\node at (10,-1.2) {$L^{*}_{4}$};

\end{tikzpicture}
\caption{Invalid Landscapes}
  \label{Invalid Landscapes}
\end{figure}

\begin{proof}
Let $\{n_1,n_2,n_3,n_4,n_5,n_6,n_7,n_8\}$ $=$ $\{1,2,3,4,5,6,7,8\}$ such that in nets of our copy of $\mathcal P_8$, $F_{n_2}$, $F_{n_4}$, and $F_{n_6}$ occur about $F_{n_1}$ in the same counter-clockwise order as $F_2$, $F_4$, and $F_6$ occur about $F_1$, and $F_{n_b}$ and $F_{n_c}$ are opposite faces given that $b+c=9$.  We state the following claim.

\vspace{0.1in}

\begin{addmargin}[0.87cm]{0cm}
\noindent {\bf Claim.} No landscape in the set $\{L_a^*\}_{a=1}^4$ depicted in Figure \ref{Invalid Landscapes} is valid.

\vspace{0.1in}

Before we prove the claim, it is worth noting that one can show through exhaustion that every landscape of the octahedron with initial face $F_{n_1}$ and second face $F_{n_2}$ not in the set $\{L_a\}_{a=1}^{9}$ contains an element of $\{L_a^*\}_{a=1}^4$ as a sublandscape; and thus, once the claim is proven, it is clear that no such landscape is valid.  Furthermore, due to the symmetries of the octahedron, any landscape with initial face $F_{n_1}$ can be mapped under rotation, and thus a reidentification of $\{n_2,n_4,n_6\} \rightarrow \{2,4,6\}$ and their corresponding opposite faces, to a landscape with initial face $F_{n_1}$ and second face $F_{n_2}$.  Finally, since path lengths on an octahedron are invariant under such reidentifications, it follows that any landscape with initial face $F_{n_1}$ not in the set $\{L_a\}_{a=1}^{9}$ must be invalid.  With this in mind, the following case by case analysis of the family of landscapes $\{L_a^*\}_{a=1}^4$ completes the proof of the theorem.

\vspace{0.1in}

\noindent {\em Proof of Claim.}

\vspace{-0.1in}

\setcounter{case}{0}

\begin{case} Landscape $L_1^*$
\end{case}

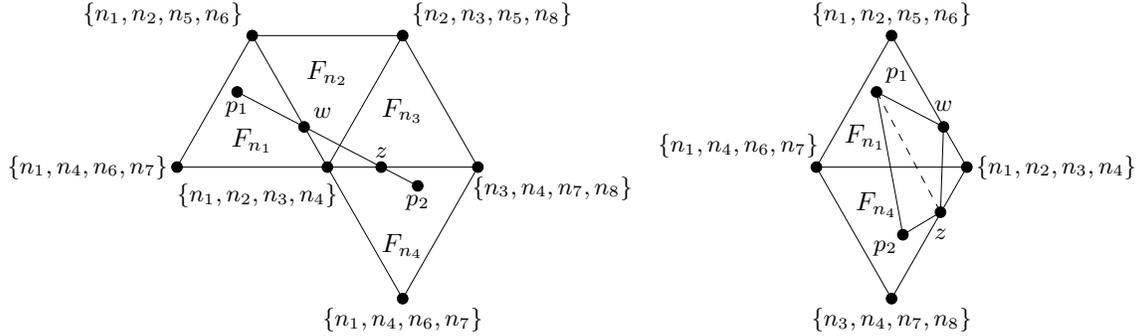
\begin{figure}[h]
  \centering
\begin{tikzpicture}[scale=1]

\filldraw[fill=black,draw=black] (6.5,1.75) circle (2pt);
\filldraw[fill=black,draw=black] (7.5,0) circle (2pt);
\filldraw[fill=black,draw=black] (8.5,1.75) circle (2pt);
\filldraw[fill=black,draw=black] (7.5,3.5) circle (2pt);
\draw[line width = 0.05mm] (6.5,1.75) -- (7.5,0) -- (8.5,1.75) -- (7.5,3.5) -- (6.5,1.75);
\draw[line width = 0.05mm] (6.5,1.75) -- (8.5,1.75);
\node[above left][font = {\small}] at (6.65,1.75) {$\{n_{1},n_{4},n_{6},n_{7}\}$};
\node[below][font = {\small}] at (7.5,0) {$\{n_{3},n_{4},n_{7},n_{8}\}$};
\node[right][font = {\small}] at (8.5,1.75) {$\{n_{1},n_{2},n_{3},n_{4}\}$};
\node[above][font = {\small}] at (7.5,3.5) {$\{n_{1},n_{2},n_{5},n_{6}\}$};
\filldraw[fill=black,draw=black] (7.3,2.75) circle (2pt);
\filldraw[fill=black,draw=black] (7.65,.85) circle (2pt);
\filldraw[fill=black,draw=black] (8.19,2.285) circle (2pt);
\filldraw[fill=black,draw=black] (8.15,1.15) circle (2pt);
\draw[line width = 0.05mm] (7.3,2.75) -- (7.65,.85);
\draw[line width = 0.05mm] (7.3,2.75) -- (8.19,2.285);
\draw[line width = 0.05mm] (7.65,.85) -- (8.15,1.15);
\draw[line width = 0.05mm] (8.19,2.285) -- (8.15,1.15);
\node[above right][font = {\small}] at (7.3,2.75) {$p_{1}$};
\node[below left, inner sep=2][font = {\small}] at (7.65,.85) {$p_{2}$};
\node[above, inner sep=6][font = {\small}] at (8.19,2.285) {$w$};
\node[below, inner sep=6][font = {\small}] at (8.15,1.15) {$z$};
\node at (7.15,2.15) {$F_{n_1}$};
\node at (7.3,1.25) {$F_{n_4}$};
\draw[line width=0.05mm][dash pattern={on 3pt}] (7.3,2.75) -- (8.15,1.15);

\filldraw[fill=black,draw=black] (0,1.75) circle (2pt);
\filldraw[fill=black,draw=black] (1,0) circle (2pt);
\filldraw[fill=black,draw=black] (2,1.75) circle (2pt);
\filldraw[fill=black,draw=black] (1,3.5) circle (2pt);
\filldraw[fill=black,draw=black] (-1,3.5) circle (2pt);
\filldraw[fill=black,draw=black] (-2,1.75) circle (2pt);
\draw[line width = 0.05mm] (0,1.75) -- (1,0) -- (2,1.75) -- (1,3.5) -- (-1,3.5) -- (-2,1.75) -- (0,1.75);
\draw[line width = 0.05mm] (0,1.75) -- (-1,3.5);
\draw[line width = 0.05mm] (0,1.75) -- (2,1.75);
\draw[line width = 0.05mm] (0,1.75) -- (1,3.5);
\node[left][font = {\small}] at (-2,1.75) {$\{n_{1},n_{4},n_{6},n_{7}\}$};
\node[above left][font = {\small}] at (-1,3.5) {$\{n_{1},n_{2},n_{5},n_{6}\}$};
\node[below left, inner sep=6][font = {\small}] at (0.35,1.75) {$\{n_{1},n_{2},n_{3},n_{4}\}$};
\node[above right][font = {\small}] at (1,3.5) {$\{n_{2},n_{3},n_{5},n_{8}\}$};
\node[below right][font = {\small}] at (1.8,1.75) {$\{n_{3},n_{4},n_{7},n_{8}\}$};
\node[below][font = {\small}] at (1,0) {$\{n_{1},n_{4},n_{6},n_{7}\}$};
\filldraw[fill=black,draw=black] (-1.2,2.75) circle (2pt);
\filldraw[fill=black,draw=black] (1.2,1.5) circle (2pt);
\filldraw[fill=black,draw=black] (-.31,2.285) circle (2pt);
\filldraw[fill=black,draw=black] (.715,1.75) circle (2pt);
\draw[line width = 0.05mm] (-1.2,2.75) -- (1.2,1.5);
\node[below][font = {\small}] at (-1.2,2.75) {$p_{1}$};
\node[below][font = {\small}] at (1.2,1.5) {$p_{2}$};
\node[above right][font = {\small}] at (-.31,2.285) {$w$};
\node[above][font = {\small}] at (.715,1.75) {$z$};
\node at (-1,2.1) {$F_{n_1}$};
\node at (0,3) {$F_{n_2}$};
\node at (1,2.5) {$F_{n_3}$};
\node at (1,0.7) {$F_{n_4}$};

\end{tikzpicture}
    \caption{Trail on $L_1^*$}
\end{figure}

Let $p_1 \in {F_{n_1}}, p_2 \in {F_{n_4}}$ and let $O_1^*$ be some orientation of $L_1^*$.  We will have two cases.

\vspace{0.1in}

\begin{addmargin}[0.87cm]{0cm}
    {\bf Subcase 1a.} {\em Either $p_1 \in \partial F_{n_1} \cap \partial F_{n_2}$ or $p_2 \in \partial F_{n_3} \cap \partial F_{n_4}$}
    
   \vspace{0.1 in}
    
  Since either $p_1 \in \partial F_{n_1} \cap \partial F_{n_2}$ or $p_2 \in \partial F_{n_3} \cap \partial F_{n_4}$, the trail is completely contained in $L_3(F_{n_1}\rightarrow F_{n_3})$ or $L_3(F_{n_2}\rightarrow F_{n_4})$. Also, since $L_3(F_{n_1}\rightarrow F_{n_3})$ and $L_3(F_{n_2}\rightarrow F_{n_4})$ are proper sublandscapes of $L_1^*$, it follows that $T_{O_1^*}(p_1,p_2)$ does not witness the validity of $L_1^*$.
   \end{addmargin}
   
 \vspace{0.1in}  
   
\begin{addmargin}[0.87cm]{0cm}
    {\bf Subcase 1b.} { $p_1 \not \in \partial F_{n_1} \cap \partial F_{n_2}$ and $p_2 \not \in \partial F_{n_3} \cap \partial F_{n_4}$}
    \vspace{0.1 in}

Now let $w = T_{O_1^*}(p_1,p_2)\cap  \overline{\{n_1,n_2,n_5,n_6\}\{n_1,n_2,n_3,n_4\}}$, $z = T_{O_1^*}(p_1,p_2)\cap \overline{\{n_1,n_2,n_3,n_4\}\{n_3,n_4,n_7,n_8\}}$, $\hat{L}=L_1(F_{n_1}\rightarrow F_{n_4})$, and $\hat{O}=(\hat{L},F_{n_1},F_{n_4})$.  Since $w \ne p_1$ and $z \ne p_2$, it follows from the triangle inequality that 
\[\abs{T_{\hat{L}}(p_1,z)}\leq\abs{T_{\hat{L}}(p_1,w)}+\abs{T_{\hat{L}}(w,z)} \hspace{0.25in} \text{and} \hspace{0.25in} \abs{T_{\hat{L}}(p_1,p_2)}<\abs{T_{\hat{L}}(p_1,z)}+\abs{T_{\hat{L}}(z,p_2)}.\]
Thus,
\[\abs{T_{\hat{L}}(p_1,p_2)}<\abs{T_{\hat{L}}(p_1,w)}+\abs{T_{\hat{L}}(w,z)}+\abs{T_{\hat{L}}(z,p_2)}.\]
We have that $\abs{T_{L_1^*}(p_1,w)}=\abs{T_{\hat{L}}(p_1,w)}$ and $\abs{T_{L_1^*}(p_2,z)}=\abs{T_{\hat{L}}(p_2,z)}$. Since $\Delta w\{n_1,n_2,n_3,n_4\}z$ in $L_1^*$ is congruent with $\Delta w\{n_1,n_2,n_3,n_4\}z$ in $\hat{L}$,
we also have that $\abs{T_{L_1^*}(w,z)}=\abs{T_{\hat{L}}(w,z)}$.

As a result, we see that
\begin{equation*}
\begin{split}
    \abs{T_{L_1^*}(p_1,p_2)}&=\abs{T_{L_1^*}(p_1,w)}+\abs{T_{L_1^*}(w,z)}+\abs{T_{L_1^*}(z,p_2)}\\
    &=\abs{T_{\hat{L}}(p_1,w)}+\abs{T_{\hat{L}}(w,z)}+\abs{T_{\hat{L}}(z,p_2)}\\
    &>\abs{T_{\hat{L}}(p_1,p_2)}.
\end{split}
\end{equation*}

\end{addmargin}
 
We have thus constructed a trail, $T_{\hat{O}}(p_1,p_2)$, on the surface of $\mathcal P_8$ between $p_1$ and $p_2$ whose length is shorter than $T_{O_1^*}(p_1,p_2)$.  Due to this, $T_{O_1^*}(p_1,p_2)$ does not witness the validity of $L_1^*$.
\vspace{0.1in}

\noindent In either case, since $p_1$ and $p_2$ were chosen arbitrarily, $L_1^*$ can not be a valid landscape.

\begin{case}
Landscape $L_2^*$
\end{case}

Due to the symmetries of $\mathcal P_8$, Case 2 is, up to rigid transformation and relabeling, identical to Case 1 and thus an immediate result of the proof of Case 1.

\begin{case} Landscape $L_3^*$
\end{case}

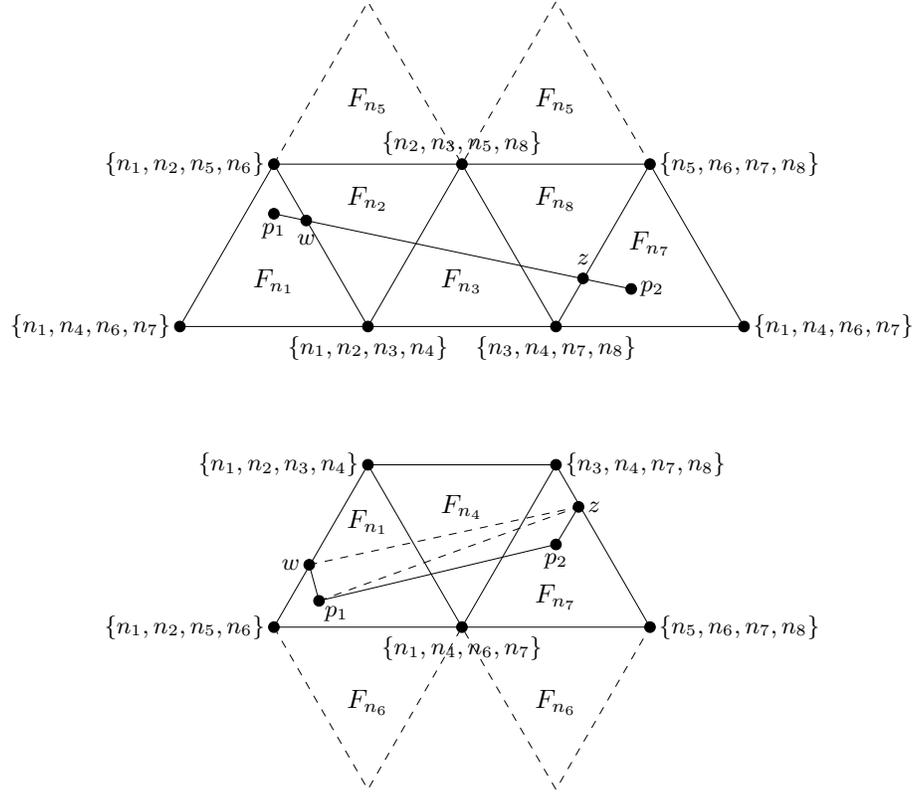
\begin{figure}[h]
  \centering
\begin{tikzpicture}[scale=1]

\filldraw[fill=black,draw=black] (0,0) circle (2pt);
\filldraw[fill=black,draw=black] (2.5,0) circle (2pt);
\filldraw[fill=black,draw=black] (5,0) circle (2pt);
\filldraw[fill=black,draw=black] (7.5,0) circle (2pt);
\filldraw[fill=black,draw=black] (1.25,2.16) circle (2pt);
\filldraw[fill=black,draw=black] (3.75,2.16) circle (2pt);
\filldraw[fill=black,draw=black] (6.25,2.16) circle (2pt);
\draw[line width = 0.05mm] (0,0) -- (1.25,2.16) -- (2.5,0) -- (3.75,2.16) -- (5,0) -- (6.25,2.16) -- (7.5,0);
\draw[line width = 0.05mm] (0,0) -- (7.5,0);
\draw[line width = 0.05mm] (1.25,2.16) -- (6.25,2.16);
\node[left][font = {\small}] at (0,0) {$\{n_{1},n_{4},n_{6},n_{7}\}$};
\node[below][font = {\small}] at (2.5,0) {$\{n_{1},n_{2},n_{3},n_{4}\}$};
\node[below][font = {\small}] at (5,0) {$\{n_{3},n_{4},n_{7},n_{8}\}$};
\node[right][font = {\small}] at (7.5,0) {$\{n_{1},n_{4},n_{6},n_{7}\}$};
\node[above, left][font = {\small}] at (1.25,2.16) {$\{n_{1},n_{2},n_{5},n_{6}\}$};
\node[above][font = {\small}] at (3.75,2.16) {$\{n_{2},n_{3},n_{5},n_{8}\}$};
\node[above, right][font = {\small}] at (6.25,2.16) {$\{n_{5},n_{6},n_{7},n_{8}\}$};
\filldraw[fill=black,draw=black] (1.25,1.5) circle (2pt);
\filldraw[fill=black,draw=black] (6,.5) circle (2pt);
\filldraw[fill=black,draw=black] (1.68,1.41) circle (2pt);
\filldraw[fill=black,draw=black] (5.36,.64) circle (2pt);
\draw[line width = 0.05mm] (1.25,1.5) -- (6,.5);
\node[below][font = {\small}] at (1.25,1.5) {$p_{1}$};
\node[right][font = {\small}] at (6,.5) {$p_{2}$};
\node[below, outer sep=2][font = {\small}] at (1.68,1.41) {$w$};
\node[above, outer sep=2][font = {\small}] at (5.36,.64) {$z$};
\node at (1.25,.6) {$F_{n_1}$};
\node at (2.5,1.7) {$F_{n_2}$};
\node at (3.75,.6) {$F_{n_3}$};
\node at (5,1.7) {$F_{n_8}$};
\node at (6.25,1.1) {$F_{n_7}$};
\draw[line width=0.05mm][dash pattern={on 3pt}] (1.25,2.16) -- (2.5,4.32) -- (3.75,2.16) -- (5,4.32) -- (6.25,2.16);
\node at (2.5,3) {$F_{n_5}$};
\node at (5,3) {$F_{n_5}$};

\filldraw[fill=black,draw=black] (1.25,-4) circle (2pt);
\filldraw[fill=black,draw=black] (3.75,-4) circle (2pt);
\filldraw[fill=black,draw=black] (6.25,-4) circle (2pt);
\filldraw[fill=black,draw=black] (2.5,-1.84) circle (2pt);
\filldraw[fill=black,draw=black] (5,-1.84) circle (2pt);
\draw[line width = 0.05mm] (1.25,-4) -- (2.5,-1.84) -- (3.75,-4) -- (5,-1.84) -- (6.25,-4) -- (1.25,-4);
\draw[line width = 0.05mm] (2.5,-1.84) -- (5,-1.84);
\node[left][font = {\small}] at (1.25,-4) {$\{n_{1},n_{2},n_{5},n_{6}\}$};
\node[below][font = {\small}] at (3.75,-4) {$\{n_{1},n_{4},n_{6},n_{7}\}$};
\node[right][font = {\small}] at (6.25,-4) {$\{n_{5},n_{6},n_{7},n_{8}\}$};
\node[left][font = {\small}] at (2.5,-1.84) {$\{n_{1},n_{2},n_{3},n_{4}\}$};
\node[right][font = {\small}] at (5,-1.84) {$\{n_{3},n_{4},n_{7},n_{8}\}$};
\filldraw[fill=black,draw=black] (1.85,-3.65) circle (2pt);
\filldraw[fill=black,draw=black] (5,-2.9) circle (2pt);
\filldraw[fill=black,draw=black] (1.72,-3.17) circle (2pt);
\filldraw[fill=black,draw=black] (5.3,-2.4) circle (2pt);
\draw[line width = 0.05mm] (1.72,-3.17) -- (1.85,-3.65) -- (5,-2.9) -- (5.3,-2.4);
\draw[line width=0.05mm][dash pattern={on 3pt}] (1.72,-3.17) -- (5.3,-2.4);
\draw[line width=0.05mm][dash pattern={on 3pt}] (1.85,-3.65) -- (5.3,-2.4);
\node[below right, inner sep=2][font = {\small}] at (1.85,-3.65) {$p_{1}$};
\node[below][font = {\small}] at (5,-2.9) {$p_{2}$};
\node[left][font = {\small}] at (1.72,-3.17) {$w$};
\node[right][font = {\small}] at (5.3,-2.4) {$z$};
\node at (2.5,-2.6) {$F_{n_1}$};
\node at (3.75,-2.4) {$F_{n_4}$};
\node at (5,-3.6) {$F_{n_7}$};
\draw[line width=0.05mm][dash pattern={on 3pt}] (1.25,-4) -- (2.5,-6.16) -- (3.75,-4) -- (5,-6.16) -- (6.25,-4);
\node at (2.5,-5) {$F_{n_6}$};
\node at (5,-5) {$F_{n_6}$};

\end{tikzpicture}
    \caption{Trail on $L_3^*$}
\end{figure}

    Let $p_1 \in {F_{n_1}}, p_2 \in {F_{n_7}}$ and let $O_3^*$ be some orientation of $L_3^*$.  Now let $w = T_{O_3^*}(p_1,p_2)\cap  \overline{\{n_1,n_2,n_5,n_6\}\{n_1,n_2,n_3,n_4\}}$ and $z = T_{O_3^*}(p_1,p_2)\cap \overline{\{n_5,n_6,n_7,n_8\}\{n_3,n_4,n_7,n_8\}}$. We will have two cases.

\vspace{0.1in}

\begin{addmargin}[0.87cm]{0cm}
    {\bf Subcase 3a.} {\em Either $w=p_1$ or $z=p_2$}
    
   \vspace{0.1 in}
    
  Since either $w=p_1$ or $z=p_2$, the trail is completely contained in $L_5(F_{n_1}\rightarrow F_{n_8})$ or $L_7(F_{n_2}\rightarrow F_{n_7})$.  Also, since $L_5(F_{n_1}\rightarrow F_{n_8})$ and $L_7(F_{n_2}\rightarrow F_{n_7})$ are proper sublandscapes of $L_3^*$, it follows that $T_{O_3^*}(p_1,p_2)$ does not witness the validity of $L_3^*$.
   \end{addmargin}
   
 \vspace{0.1in}  
   
\begin{addmargin}[0.87cm]{0cm}
    {\bf Subcase 3b.} { $w \ne p_1$ and $z \ne p_2$}
    \vspace{0.1 in}

Let $\hat{L}=L_3(F_{n_1}\rightarrow F_{n_7})$ and $\hat{O}=(\hat{L},F_{n_1},F_{n_6})$. 
Since $w \ne p_1$ and $z \ne p_2$, it follows from the triangle inequality that 
\[\abs{T_{\hat{L}}(p_1,z)}<\abs{T_{\hat{L}}(p_1,w)}+\abs{T_{\hat{L}}(w,z)} \hspace{0.25in} \text{and} \hspace{0.25in} \abs{T_{\hat{L}}(p_1,p_2)}<\abs{T_{\hat{L}}(p_1,z)}+\abs{T_{\hat{L}}(z,p_2)}.\]
Thus,
\[\abs{T_{\hat{L}}(p_1,p_2)}<\abs{T_{\hat{L}}(p_1,w)}+\abs{T_{\hat{L}}(w,z)}+\abs{T_{\hat{L}}(z,p_2)}.\]

Let $w(L_3^*, F_{n_2}, F_{n_5})=(x_{1}, y_{1})$, which can be represented as $w(L_3^*, F_{n_8}, F_{n_5})=(x_{1}+1, y_{1})$. Let $z(L_3^*, F_{n_8}, F_{n_5})=(x_{2}, y_{2})$. We get that
\[\abs{T_{L_3^*}{(w,z)}}=\sqrt{(x_{2}-1-x_{1})^2+(y_{2}-y_{1})^2}.\]

Since $\Delta \{n_1,n_2,n_5,n_6\}w\{n_2,n_3,n_5,n_8\}$ in $L_3^*$ is congruent with $\Delta \{n_1,n_2,n_5,n_6\}w\{n_1,n_4,n_6,n_7\}$ in $\hat{L}$, $w(L_3^*, F_{n_2}, F_{n_5})=(x_{1}, y_{1})$ can be represented as $w(\hat{L}, F_{n_1}, F_{n_6})=(1-x_{1}, y_{1})$. Since $\Delta \{n_2,n_3,n_5,n_8\}z\{n_5,n_6,n_7,n_8\}$ in $L_3^*$ is congruent with $\Delta \{n_1,n_4,n_6,n_7\}z\{n_5,n_6,n_7,n_8\}$ in $\hat{L}$, $z(L_3^*, F_{n_8}, F_{n_5})=(x_{2}, y_{2})$ can be represented as $z(\hat{L}, F_{n_7}, F_{n_6})=(1-x_{2}, y_{2})$, and it can also be represented as $z(\hat{L}, F_{n_1}, F_{n_6})=(2-x_{2}, y_{2})$. We get that
\[\abs{T_{\hat{L}}(w,z)}=\sqrt{(x_{1}+1-x_{2})^2+(y_{1}-y_{2})^2}.\]

It follows that $\abs{T_{L_3^*}(w,z)}=\abs{T_{\hat{L}}(w,z)}$. We also have that $\abs{T_{L_3^*}(p_1,w)}=\abs{T_{\hat{L}}(p_1,w)}$ and $\abs{T_{L_3^*}(p_2,z)}=\abs{T_{\hat{L}}(p_2,z)}$.
As a result, we see that
\begin{equation*}
\begin{split}
    \abs{T_{L_3^*}(p_1,p_2)}&=\abs{T_{L_3^*}(p_1,w)}+\abs{T_{L_3^*}(w,z)}+\abs{T_{L_3^*}(z,p_2)}\\
    &=\abs{T_{\hat{L}}(p_1,w)}+\abs{T_{\hat{L}}(w,z)}+\abs{T_{\hat{L}}(z,p_2)}\\
    &>\abs{T_{\hat{L}}(p_1,p_2)}.
\end{split}
\end{equation*}

\end{addmargin}
 
We have thus constructed a trail, $T_{\hat{O}}(p_1,p_2)$, on the surface of $\mathcal P_8$ between $p_1$ and $p_2$ whose length is shorter than $T_{O_3^*}(p_1,p_2)$.  Due to this, $T_{O_3^*}(p_1,p_2)$ does not witness the validity of $L_3^*$.
\vspace{0.1in}

\noindent In either case, since $p_1$ and $p_2$ were chosen arbitrarily, $L_3^*$ can not be a valid landscape.

\begin{case}
Landscape $L_4^*$
\end{case}

Due to the symmetries of $\mathcal P_8$, Case 4 is, up to rigid transformation and relabeling, identical to Case 3 and thus an immediate result of the proof of Case 3.

\vspace{0.1in}

\end{addmargin}

\noindent Having shown that no landscape in the family $\{L_a^*\}_{a=1}^4$ is valid, we have thus completed our proof. 
\end{proof}

\begin{corollary}
The octahedron has $9$ valid landscapes $\{L_a\}_{a=1}^{9}$.
\end{corollary}

\begin{proof}
One can confirm via a short computation that the following table provides pairs of points which witness the validity of the landscapes $L_1, L_2, L_3, L_4, L_5, L_6, L_7, L_8$, and $L_9$.  Thus all nine of these landscapes are valid.  However, as shown by Theorem \ref{finaltheorem}, the octahedron has no other valid landscapes.
\begin{center}
\begin{tabular}{ |c|c|}
    \hline
    Landscape & Pair of Points Witnessing Validity\\
    \hline
    $L_1$ & \rule{0pt}{10pt}$(F_{n_1}, F_{n_2}, 0.5, 0.2),(F_{n_2}, F_{n_1}, 0.5, 0.2)$\\[2pt]
    \hline
    $L_2$ & \rule{0pt}{10pt}$(F_{n_1}, F_{n_2}, 0.5, 0.1), (F_{n_5}, F_{n_6}, 0.8, 0.1)$ \\[2pt]
    \hline
    $L_3$ & \rule{0pt}{10pt}$(F_{n_1}, F_{n_2}, 0.8, 0.1), (F_{n_5}, F_{n_6}, 0.5, 0.1)$ \\[2pt]
    \hline
    $L_4$& \rule{0pt}{10pt}$(F_{n_1}, F_{n_2}, 0.9, \frac {1}{6}), (F_{n_8}, F_{n_7}, 0.9, \frac {1}{6})$\\[2pt]
    \hline
    $L_5$ & \rule{0pt}{10pt}$(F_{n_1}, F_{n_2}, 0.1, 0.1), (F_{n_8}, F_{n_7}, 0.4, \frac {2}{3})$ \\[2pt]
    \hline
    $L_6$ & \rule{0pt}{10pt}$(F_{n_1}, F_{n_2}, 0.4, \frac {2}{3}), (F_{n_8}, F_{n_7}, 0.1, 0.1)$ \\[2pt]
    \hline
    $L_7$ & \rule{0pt}{10pt}$(F_{n_1}, F_{n_2}, 0.6, \frac {2}{3}), (F_{n_8}, F_{n_7}, 0.9, 0.1)$ \\[2pt]
    \hline
    $L_8$ & \rule{0pt}{10pt}$(F_{n_1}, F_{n_2}, 0.9, 0.1), (F_{n_8}, F_{n_7}, 0.6, \frac {2}{3})$ \\[2pt]
    \hline
    $L_9$ & \rule{0pt}{10pt}$(F_{n_1}, F_{n_2}, 0.1, \frac {1}{6}), (F_{n_8}, F_{n_7}, 0.1, \frac {1}{6})$ \\[2pt]
    \hline
\end{tabular}
\end{center}
\end{proof}

Having identified the collection of valid landscapes of $\mathcal P_8$, we provide the following corollary which identifies a way to determine the surface distance between any two points on a octahedron.

\begin{corollary}
Let $p_1,p_2$ be two distinct points with $p_1 \in F_n$ and $p_2 \in F_m$, for $F_n$ and $F_m$ distinct faces of the octahedron.
\begin{itemize}
    \item If $F_n$ and $F_m$ are adjacent, then
    \[d_{\mathcal P_8}(p_1,p_2) = \abs{T_{L_1}(p_1, p_2)}.\]
    \item If $F_n$ and $F_m$ are neither adjacent nor opposite, then
    \[d_{\mathcal P_8}(p_1,p_2) = \min\left\{\abs{T_{L_2}(p_1, p_2)}, \abs{T_{L_3}(p_1, p_2)}\right\}.\]
    \item If $F_n$ and $F_m$ are opposite, then \[d_{\mathcal P_8}(p_1,p_2) = \min\left\{\abs{T_{L_a}(p_1, p_2)}: 4 \leq a \leq 9\right\}.\]
\end{itemize}
\end{corollary}

Having identified the surface distance of points on the octahedron, we have reached the conclusion of the current discussion.  However, as shown in \cite{Fontenot}, the concepts developed here can be applied to any convex unit polyhedron, and thus much work remains to be done. In particular, this includes the remaining platonic solids.  Finally, a question that is worth asking is if this concept can be applied to any convex polyhedron or perhaps even certain classes of nonconvex polyhedra.

\section*{Acknowledgements}

The research of Emiko Saso and Xin Shi was supported by the Trinity College Summer Research Program.

\bibliographystyle{abbrv}
\bibliography{octahedron}

\begin{thebibliography}{1}

\bibitem{Agarwal}
P.~K. Agarwal, B.~Aronov, J.~O'Rourke, and C.~A. Schevon.
\newblock Star unfolding of a polytope with applications.
\newblock {\em SIAM Journal on Computing}, 26(6):1689--1713, 1997.

\bibitem{Alexandrov}
A.~D. Aleksandrov and K.~Polyeder.
\newblock {\em Math. Lehrbucher und Monographien.}
\newblock Akademie-Verlag, Berlin, 1958.

\bibitem{Chen}
J.~Chen and Y.~Han.
\newblock Shortest paths on a polyhedron.
\newblock In {\em Proceedings of the sixth annual symposium on Computational
  geometry}, pages 360--369, 1990.

\bibitem{Cook}
A.~F. Cook and C.~Wenk.
\newblock Shortest path problems on a polyhedral surface.
\newblock In {\em Algorithms and Data Structures: 11th International Symposium,
  WADS 2009, Banff, Canada, August 21-23, 2009. Proceedings 11}, pages
  156--167. Springer, 2009.

\bibitem{Fontenot}
K.~Fontenot, E.~Raign, A.~Sangalli, E.~Saso, H.~Schuerger, X.~Shi, and
  E.~Striff-Cave.
\newblock Landscapes of the tetrahedron and cube: An exploration of shortest
  paths on polyhedra.
\newblock {\em arXiv preprint arXiv:2201.04253}, 2023.

\bibitem{Kanai}
T.~Kanai and H.~Suzuki.
\newblock Approximate shortest path on a polyhedral surface based on selective
  refinement of the discrete graph and its applications.
\newblock In {\em Proceedings Geometric Modeling and Processing 2000. Theory
  and Applications}, pages 241--250. IEEE, 2000.

\bibitem{Magnanti}
T.~L. Magnanti and P.~Mirchandani.
\newblock Shortest paths, single origin-destination network design, and
  associated polyhedra.
\newblock {\em Networks}, 23(2):103--121, 1993.

\end{thebibliography}

\end{document}